\documentclass[a4paper,fleqn,review]{cas-sc}

\usepackage[numbers,sort&compress]{natbib}
\usepackage{mathtools,amssymb,graphicx,physics,amsmath,amsfonts,mathrsfs}
\usepackage{bm}
\usepackage{booktabs,multirow,array}
\usepackage{xcolor}
\usepackage[nice]{nicefrac}
\usepackage{cleveref}
\usepackage{amsthm}

\newtheorem{lemma}{Lemma}[section]
\newtheorem{theorem}{Theorem}[section]
\theoremstyle{remark}
\newtheorem{remark}{Remark}[section]

\newcommand{\jmp}[1]{[\![#1]\!]}        

\def\calT{{\mathcal T}}

\renewcommand{\qedsymbol}{}

\begin{document}
	\let\WriteBookmarks\relax
	\def\floatpagepagefraction{1}
	\def\textpagefraction{.001}
	
	\shorttitle{Conservative DG via Implicit Penalization for gKdV and HS--KdV}
	\shortauthors{M.S. Tariq et~al.}
	
	\title [mode = title]{High-order Conservative Discontinuous Galerkin Methods via Implicit Penalization for the Generalized Korteweg--de Vries Equation and the Hirota--Satsuma KdV System}
	
	\author[1]{M. Shan Tariq}[orcid=0009-0006-9831-8966]
	\ead{mtariq4@umassd.edu}
	\credit{Conceptualization, Methodology, Software, Writing - Original Draft}
	
	\author[1]{Yanlai Chen}
	\ead{yanlai.chen@umassd.edu}
	\credit{Conceptualization, Writing - Review \& Editing}
	
	\author[1]{Bo Dong}
    \cormark[1]
	\ead{bdong@umassd.edu}
	\credit{Conceptualization, Methodology, Supervision, Writing - Review \& Editing}

	\affiliation[1]{organization={Department of Mathematics, University of Massachusetts Dartmouth},
		addressline={285 Old Westport Road}, 
		city={North Dartmouth},
		state={MA},
		postcode={02747-2300}, 
		country={USA}}
	
	\cortext[cor1]{Corresponding author}
	
	\begin{abstract}
		We develop new conservative discontinuous Galerkin (DG) methods for nonlinear wave problems, focusing on the generalized Korteweg--de Vries (gKdV) equation and the coupled Hirota--Satsuma KdV (HS-KdV) system. The proposed methods preserve mass through the single-valued structure of numerical traces, while energy and Hamiltonian conservation are enforced by implicitly determining penalty parameters in the numerical traces through auxiliary conservation constraints. In our previous work \cite{ChenDongPereira2022}, we developed a conservative DG method for the gKdV equation; however, that formulation involves the time derivative of the jump of the approximate solution, which complicates extensions beyond second-order temporal accuracy. Our new formulation overcomes this limitation by introducing a redesigned trace configuration that eliminates the derivative-of-jump term. This novel enhancement seamlessly paves the way for higher-order time discretizations and requires solving fewer nonlinear systems per time step than the previous approach. For the coupled HS-KdV system, we present the first conservative DG method that preserves all three invariants of the exact solution. Numerical results demonstrate the accuracy and expected convergence behavior of the proposed methods, as well as long-time stability and strong conservation properties for both the gKdV equation and HS-KdV system.
	\end{abstract}
	
	
\maketitle

\section{Introduction}

In this paper, we study structure-preserving discontinuous Galerkin (DG) methods for two closely related nonlinear dispersive models: the generalized Korteweg--de Vries (gKdV) equation and the coupled Hirota--Satsuma Korteweg--de Vries (HS-KdV) system. The two developments are unified through a common philosophy of enforcing multiple conservation laws at the discrete level via suitably designed numerical traces and implicit penalization.

The gKdV equation is given by the following third-order nonlinear dispersive equation:
\begin{equation}\label{eq:gkdv}
    u_t+\epsilon u_{xxx}+\partial_x f(u)=g(x,t),
    \qquad x\in\Omega=[a,b], \quad t>0,
\end{equation}
with periodic boundary conditions and the initial condition $u(x,0)=u_0(x)$. 
Here, \(f(u)\) is typically a polynomial in \(u\). In the particular case \(\epsilon=1\), \(f(u)=3u^2\), and \(g\equiv0\), Eq. \eqref{eq:gkdv} reduces to the classical KdV equation. The KdV equation and its generalized variants are fundamental nonlinear dispersive models arising in the study of shallow-water waves, internal waves, plasma physics, nonlinear optics, and related areas of applied mathematics and mathematical physics \cite{KortewegDeVries1895,HammackSegur1974,HammackSegur1978,GardnerMorikawa1960,HelfrichWhitehead1990}. Due to the interplay between nonlinearity and dispersion, these equations support solitary waves, cnoidal waves, and other coherent structures whose long-time dynamics have been studied extensively from both analytical and numerical viewpoints.

For the homogeneous problem with \(g\equiv0\), the exact solution of \eqref{eq:gkdv} preserves several basic invariants, among which the most important are mass, $L^2$ energy and Hamiltonian:
\[
	\mathcal{M}(u) = \int_\Omega u\,dx, \qquad
	\mathcal{E}(u) = \int_\Omega u^2\,dx, \qquad
	\mathcal{H}(u) = \int_\Omega \left(\frac{\epsilon}{2}u_x^2 - V(u)\right)\,dx,
\]
where \(V(\cdot)\) is an antiderivative of \(f(\cdot)\). 
These invariants play a central role in the qualitative behavior of solutions and motivate the construction of structure-preserving numerical schemes.

A broad range of numerical methods has been developed for KdV-type equations, including finite difference methods \cite{Vliegenthart1971,Goda1977,LiVisbal2006}, finite element methods \cite{Winther1980,ArnoldWinther1982,SanzSernaChristie1981,BakerDougalisKarakashian1983}, spectral and pseudospectral methods \cite{FornbergWhitham1978,HuangSloan1992,GuoShen2001,MaSun2001}, and operator-splitting methods \cite{HoldenKarlsenRisebro1999,HoldenKarlsenRisebroTao2011}. Among high-order spatial discretizations, discontinuous Galerkin (DG) and local discontinuous Galerkin (LDG) methods are especially attractive because they combine high-order accuracy with compact stencils, local conservation, flexibility on nonuniform meshes, and the ability to vary the polynomial degree locally \cite{YanShu2002,XuShu2005,XuShu2012,Shu2009}.

For gKdV equations, a substantial body of work has focused on invariant-preserving DG discretizations. Bona \textit{et al.}~\cite{BonaChenKarakashianXing2013}, Yi \textit{et al.}~\cite{YiHuangLiu2013}, Karakashian and Xing \cite{karakashian2016posteriori}, and Fu and Shu \cite{FuShu2019} developed conservative DG or LDG formulations that preserve mass and the \(L^2\) energy. Liu and Yi \cite{liu2016hamiltonian} introduced a Hamiltonian-preserving DG method for the generalized KdV equation, while Zhang and Xia \cite{zhang2019conservative} proposed two LDG variants, one preserving the \(L^2\) energy and the other preserving the Hamiltonian. The \(H^2\)-conservative DG method of Chen, Cockburn, and Dong \cite{ChenCockburnDong2016}, although derived for third-order linear equations rather than the nonlinear gKdV equation itself, is also an important methodological precursor in the development of conservative DG discretizations for odd-order dispersive problems. More recent developments include comparative studies of DG formulations for gKdV blow-up and soliton dynamics \cite{HongWeiZhao2024}, IMEX-LDG analyses for linearized KdV equations \cite{WangTaoShuZhang2024}, structure-preserving LDG methods for stochastic KdV equations \cite{LiuYangMaDing2024}, and bound-preserving LDG discretizations for KdV-type models \cite{BiZhao2025}.

These studies reveal a fundamental limitation of classical DG formulations: they typically preserve mass together with either the \(L^2\) energy or the Hamiltonian, but not both simultaneously, which leaves open the question of fully conservative discretizations for gKdV-type equations.
A notable advance was made by Chen, Dong, and Pereira \cite{ChenDongPereira2022}, who proposed an implicitly penalized DG method for the generalized KdV equation in which stabilization parameters are treated as additional unknowns and are determined through auxiliary conservation constraints. Their formulation appears to be the first DG method that simultaneously preserves mass, \(L^2\) energy, and Hamiltonian for the generalized KdV equation.

Despite its strong conservation properties, the scheme presented in \cite{ChenDongPereira2022} poses certain implementation difficulties. In particular, the discrete Hamiltonian relation contains a term involving the time derivative of the jump of the numerical solution. This complicates the fully discrete formulation and limits the use of higher-order implicit time integrators. 
Motivated {to address this limitation}, we revisit the admissible structure of numerical traces for Hamiltonian conservation and propose a revised numerical trace configuration. The new construction eliminates the time-derivative-of-jump term while preserving all conservation properties, and is naturally compatible with higher-order implicit time discretizations. The resulting scheme reduces the nonlinear coupling at each time step and leads to a more efficient fully discrete implementation, while retaining the fully conservative philosophy of \cite{ChenDongPereira2022}.

We also consider the Hirota--Satsuma coupled KdV system
\begin{equation}\label{eq:hskdv}
\begin{aligned}
    u_t &= a\left(u_{xxx}+6uu_x\right)+2bvv_x,\\
    v_t &= -v_{xxx}-3uv_x,
\end{aligned}
\end{equation}
on interval \(x\in\Omega=[x_L, x_R]\) 
with periodic boundary conditions and initial conditions $u(x,0)=u_0(x)$ and $v(x,0)=v_0(x)$.
The HS--KdV system, originally introduced by Hirota and Satsuma \cite{hirota1981soliton}, models the interaction of long waves with different dispersion relations and occupies an important place in the theory of integrable systems. Beyond the original formulation, the HS--KdV system and its generalizations have been studied extensively from the analytical point of view, including integrability tests, Lax-pair constructions, Miura transformations, exact traveling-wave solutions, and multi-component extensions \cite{Sakovich1999,WuEtAl1999,ChenEtAl2014}. In the normalization adopted in this paper, the system possesses three basic conserved quantities
\begin{align*}
\mathcal{M}(u)=\int_\Omega u\,dx,\qquad \mathcal{E}(u)=\int_\Omega \left(u^2+\frac{2}{3}bv^2\right)\,dx,\qquad \mathcal{H}(u)=\int_\Omega \left((1+a)(u^3-\frac{1}{2}u_x^2)+b(uv^2-v_x^2)\right)\,dx.
\end{align*}
These invariants  play a role analogous to those of the scalar gKdV equation and provide the basis for structure-preserving discretizations.

A variety of numerical methods has been proposed for HS--KdV-type systems, including spectral or pseudospectral collocation methods \cite{DarvishiEtAl2007}, radial basis function methods \cite{IslamEtAl2009,SadeeqEtAl2022}, decomposition and differential-transform methods \cite{Raslan2004,ZuoZhang2011}, spline-based collocation methods \cite{RaslanEtAl2016}, and finite element methods \cite{YagmurluEtAl2019}. DG methods have been developed for Ito-type KdV  systems \cite{XuShu2006,BaharloueiEtAl2022}, although these approaches typically do not enforce full invariant preservation.

Despite these developments, DG methods that simultaneously preserve all three invariants for the HS--KdV system appear to be absent in the current literature. 
Motivated to fill this gap and extend our high-order conservative DG scheme from scalar to systems, we develop a conservative DG method for the HS--KdV system within the implicit-penalization framework of~\cite{ChenDongPereira2022}. By treating penalization parameters as additional unknowns and enforcing conservation constraints, we construct a DG scheme that simultaneously preserves the discrete mass, energy, and Hamiltonian of the HS--KdV system. This yields a structure-preserving approximation that is well suited for long-time simulation of nonlinear wave interactions.

Overall, this paper presents a unified structure-preserving DG framework for both scalar and coupled KdV-type equations, based on numerical traces with penalty parameters that are implicitly determined through conservation constraints in the DG formulation. This framework enables consistent enforcement of multiple invariants across different dispersive models within a single discretization philosophy.

The remainder of the paper is organized as follows. In Section \ref{sec:kdv}, we discuss the conservative DG method for the gKdV equation, beginning with a review of the fully conservative DG framework from \cite{ChenDongPereira2022}. We then define a revised set of numerical traces and derive the improved DG formulation. Section \ref{sec:hs-kdv} is devoted to the derivation of the conservative DG formulation for the HS--KdV system and the proof of the conservation properties. In Section \ref{sec:numerics}, we present numerical results that confirm the accuracy, convergence rates, and long-time conservative behavior of the proposed methods for the gKdV equation and the HS-KdV system. Finally, Section \ref{sec:conclusion} provides concluding remarks and discusses future work. Implementation details, including the relevant flowchart, are provided in the Appendix.

\section{Conservative DG Method for the gKdV Equation} \label{sec:kdv}

In this section, we consider the gKdV equation. We begin by establishing notation and recalling the baseline DG formulation, followed by a review of the previous numerical trace design and the resulting nonlinear discrete system. The limitations of the earlier approach, particularly in terms of temporal discretization and computational complexity, are discussed.

We then introduce a new choice of numerical traces and derive the corresponding discrete system. The revised formulation eliminates the restrictive term present in the previous scheme, thereby allowing the implementation of higher-order implicit Runge--Kutta methods, specifically IRK4. Moreover, the new system significantly reduces the number of nonlinear solves required per time step, improving overall computational efficiency.

\subsection{Notation}
We partition the domain $\Omega=(a,b)$ into 
$
\mathcal{T}_h=\{I_i:=(x_{i-1},x_i): {i=1, \cdots, N}\},
$
with $x_0=a, x_N=b, h_i=x_i-x_{i-1}$ and $h=\max_{1\le i\le N} h_i$. Let
$
\partial \mathcal{T}_h:=\{\partial I_i: i=1,\ldots,N\}
$ and
$
E_h:=\{x_i\}_{i=0}^N
$
denote the set of all element boundaries and the set of all mesh nodes, 
respectively. 
For any function $\zeta\in L^2(\partial\mathcal{T}_h)$, 
we denote its values on
$
\partial I_i=\{x_{i-1}^+,x_i^-\}
$
by 
$\zeta_{i-1}^+$ 
and $\zeta_i^-$. In general, the two traces at a node need not coincide. 
In contrast, for any function $\eta\in L^2(E_h)$, its value at $x_i$ is 
uniquely defined, and we write
$
\eta_i:=\eta(x_i).
$
We use the following notation for the broken inner products:
\[
(\phi,v):=\sum_{i=1}^N \int_{I_i}\phi v\,dx,
\qquad
\langle \phi,vn\rangle:=\sum_{i=1}^N \left(\phi(x_i^-)v(x_i^-)n(x_i^-)
+
\phi(x_{i-1}^+)v(x_{i-1}^+)n(x_{i-1}^+)\right),
\]
where
 $n$ is the outward unit normal vector ($n(x_{i-1}^+)=-1, n(x_i^-)=1$).
For any piecewise function $\phi$, its average and jump at $x_i$ are 
defined by
\[
\{\phi\}(x_i):=\frac{1}{2}\bigl(\phi(x_i^-)+\phi(x_i^+)\bigr),
\qquad
\jmp{\phi}(x_i):=\phi(x_i^-)-\phi(x_i^+).
\]
We also introduce the {periodic} finite element space 
\[
W_h^k
=
\left\{
\omega\in L^2(\mathcal{T}_h):
\omega|_K\in P^k(K)\ \text{for all }K\in\mathcal{T}_h,
\ \text{and }\omega(a)=\omega(b)
\right\},
\]
where $P^k(K)$ denotes the space of polynomials of degree at most $k$ 
on $K$, and the condition $\omega(a)=\omega(b)$ enforces periodicity.
Finally, the $H^s(D)$-norm is denoted by $\|\cdot\|_{s,D}$. We omit 
the first subscript when $s=0$, and omit the second one when $D=\Omega$ 
or $\mathcal{T}_h$.

\subsection{The Baseline DG Formulation}
To construct the DG formulation, we first rewrite Eq.~\eqref{eq:gkdv} as an equivalent system of first-order equations:
\begin{align*}
q-u_x = 0, \qquad
p-\epsilon q_x = f(u), \qquad 
u_t+p_x = g(x,t) 
\end{align*}
with the initial condition \(u(x,0) = u_0(x)\) and periodic boundary conditions.
The semi-discrete DG scheme seeks approximations 
$(u_h, q_h, p_h) \in \left[W_h^k\right]^3$ that satisfy
\begin{subequations}\label{DEqs}
\begin{align}
    (q_h, v) + (u_h, v_x) 
    - \langle \hat{u}_h, v n \rangle &= 0,  
    \label{2.4a} \\
    (p_h, z) + \epsilon (q_h, z_x) 
    - \epsilon \langle \hat{q}_h, z n \rangle 
    &= (f(u_h), z), 
    \label{2.4b} \\
    (u_{h t}, w) - (p_h, w_x) 
    + \langle \hat{p}_h, w n \rangle 
    &= (g, w),
\end{align}
\end{subequations}
for any $v, z, w \in W_h^k$. Here, $\hat{u}_h$, $\hat{q}_h$, and $\hat{p}_h$ denote the numerical traces, whose definitions are crucial for the accuracy and stability of the DG method \cite{arnold2002unified}. {Their choices also play} a central role in preserving physically relevant properties of the scheme. In particular, conservation properties can be enforced through appropriate choices of numerical traces combined with additional constraints. 
We begin with a Lemma from \cite{ChenDongPereira2022}, which gives sufficient conditions for DG methods   \eqref{DEqs} to conserve  mass, $L^2$ energy, and Hamiltonian.
\begin{lemma}\label{lemma:conserv_general}
	Suppose $(u_h, q_h, p_h)$ satisfy \eqref{DEqs} with $g=0$. \\
	(i) If $\widehat{p_h}$ is single-valued, then we have
	\[
	\frac{d}{dt}\int_{\calT_h} u_h \,dx=0
	\quad \, {\rm (mass~conservation)}.
	\]
	(ii)	If $\widehat{u_h}, \widehat{q_h}, \widehat{p_h}$ are single-valued and satisfy the condition
	\begin{equation}
		\label{eq:conserv_cond1}
	\sum_{i=1}^N  \Big( (\widehat{u_h} - \{u_h\})(\jmp{\Pi f(u_h)} - \jmp{p_h})+\varepsilon(\widehat{q_h} - \{q_h\})\jmp{q_h}-(\widehat{p_h} - \{p_h\})\jmp{u_h}  +\jmp{V(u_h)} - \{\Pi f(u_h)\} \jmp{u_h} \Big)(x_i)=0,
	\end{equation}
	where $\Pi$ is the $L^2$ orthogonal projection onto $W^k_h$, then we have 
	\[ 
		\frac{d}{dt}\int_{\calT_h} u_h^2 \,dx=0
		 \quad \, {\rm (energy~ conservation)}.
	\]
	(iii)	If $\widehat{u_h}, \widehat{q_h}, \widehat{p_h}$ are single-valued and satisfy the condition
	\begin{equation}
		\label{eq:conserv_cond2}
		 \sum_{i=1}^{N} \left( \varepsilon(\widehat{u_h} - \{u_h\})_t \jmp{q_h}  +\varepsilon(\widehat{q_h} - \{q_h\}) \jmp{u_h}_t  + (\widehat{p_h} - \{p_h\})\jmp{p_h}\right)(x_i)=0,
	\end{equation}
	then we have
	\begin{equation*}
		\frac{d}{dt}\int_{\calT_h} \Big(\,\frac{\varepsilon}{2}
		q_h^2 - V(u_h)\Big) \,dx=0
		 \quad\, ({\rm Hamiltonian~conservation}).
	\end{equation*}
\end{lemma}

\subsection{The Previous Conservative DG and its Limitations}

The numerical traces used in the previous method \cite{ChenDongPereira2022} are chosen in a form analogous to that employed in local discontinuous Galerkin methods \cite{arnold2002unified}, namely,
\begin{equation*}
    \hat{u}_h = \{u_h\}, \quad
    \hat{q}_h = \{q_h\} + \tau_{qu}\jmp{u_h}, \quad
    \hat{p}_h = \{p_h\} + \tau_{pu}\jmp{u_h}. 
\end{equation*}
Instead of prescribing the penalty parameters $(\tau_{qu},\tau_{pu})$ a priori as in standard LDG formulations, they are treated here as {\it additional unknowns} to be determined. This additional flexibility makes it possible to impose two extra constraints, thereby enabling the scheme to conserve mass,  $L^2$ energy, and Hamiltonian. 

Using these traces, the discrete energy and Hamiltonian conservation conditions in Lemma \ref{lemma:conserv_general} are given by:
\begin{subequations}
	\begin{align}
		&\tau_{pu}\sum_{i=1}^N \jmp{u_h}^2(x_i) - \varepsilon\tau_{qu}\sum_{i=1}^N \jmp{u_h} \jmp{q_h}(x_i)=\sum_{i=1}^N \Big(\jmp{V(u_h)} - \{\Pi f(u_h)\} \jmp{u_h}\Big)(x_i), 	\label{eq:E_conserv}\\  
		&\tau_{pu}\sum_{i=1}^N \jmp{p_h} \jmp{u_h}(x_i)+\varepsilon\tau_{qu}\sum_{i=1}^N\jmp{u_h}_t\jmp{u_h}(x_i)=0. \label{eq:H_conserv}
    \end{align}
\end{subequations}

The gKdV equation contains a third-order spatial derivative term, which typically imposes severe time-step restrictions when explicit schemes are used. To avoid restrictive time-step constraints, implicit time-integration methods are often employed. In particular, we consider the second-order Implicit Midpoint method and a fourth-order implicit Runge--Kutta method (IRK4) \cite{BonaChenKarakashianXing2013,ChenCockburnDong2016}. These methods are widely used for dispersive and Hamiltonian-type problems due to their favorable long-time stability properties. 

In the previous work \cite{ChenDongPereira2022}, the Implicit Midpoint method was adopted for time discretization. This choice is closely tied to the structure of the Hamiltonian relation \eqref{eq:H_conserv}, which contains a time-derivative-of-jump term $\jmp{u_h}_t$. A consistent discretization of this term requires at least second-order temporal accuracy at all time levels, including the initial step.

To illustrate this issue, we briefly recall the Implicit Midpoint method used in \cite{ChenDongPereira2022}. Let \( 0 = t_0 < t_1 < \dots < t_M = T \) denote a uniform partition of the time interval \( [0,T] \) with time step \( \Delta t = t_{n+1} - t_n \). For each time level \( n = 0, \dots, M-1 \), the solution \( u_h^{n+1} \in W_h^k \) is computed via
\begin{equation}\label{2.11}
    u_h^{n+1} = 2u_h^{n+\frac{1}{2}} - u_h^n,
\end{equation}
where the intermediate state \( u_h^{n+\frac{1}{2}} \in W_h^k \) is the DG solution to the equation
\begin{equation}\label{2.12}
    \frac{u-u_h^n}{\frac{1}{2}\Delta t} + \epsilon u_{xxx} + f(u)_x = g(x, t_{n+\frac{1}{2}}).
\end{equation}
At each time step, one first solves a coupled nonlinear system for the DG variables and the penalty parameters at the intermediate stage \( t_{n+\frac{1}{2}} \). The solution $u_h$ then is updated to the next time level $t_{n+1}$. Subsequently, $q_h^{n+1}$ is obtained from the linear relation \eqref{2.4a}, while  $p_h^{n+1}$ and penalty parameters ($\tau_{qu}^{n+1}, \tau_{pu}^{n+1}$) require solving an additional nonlinear system consisting of \eqref{2.4b}, \eqref{eq:E_conserv} and \eqref{eq:H_conserv}.

It is important to note that only the IRK2 (Implicit  Midpoint) scheme was implemented in the previous work \cite{ChenDongPereira2022}. This limitation stems directly from the presence of the term $\jmp{u_h}_t$ in the Hamiltonian conservation equation \eqref{eq:H_conserv}, which introduces a temporal coupling that is not naturally compatible with standard multi-stage IRK methods.  In particular, for methods such as IRK4, a higher-order discrete approximation of $\jmp{u_h}_t$ would be required at the initial time step. Constructing such a self-starting high-order approximation is nontrivial, thereby preventing a straightforward extension to higher-order IRK schemes.

\subsection{{A New High-order} Conservative DG Method}

To overcome the limitation of the previous method in \cite{ChenDongPereira2022}, we revisit the sufficient conditions for energy and Hamiltonian conservation stated in Lemma~\ref{lemma:conserv_general}. In particular, Eq.~\eqref{eq:conserv_cond2} shows that the presence of time-derivative-of-jump terms in the conservation constraints is linked to the choice of numerical traces. 
This observation suggests that eliminating such terms requires restricting the form of the numerical traces. In particular, the traces $\widehat{u}_h$ and $\widehat{q}_h$ should not contain penalty contributions.

Motivated by this observation, we redesign the numerical traces as follows:
\begin{equation}\label{NTU}
		\hat{u}_h = \{ u_h \}, \qquad
		\hat{q}_h = \{ q_h \}, \qquad
		\hat{p}_h = \{ p_h \}  + \tau_{pu} \jmp{u_h} + \tau_{pq} \jmp{q_h}. 
\end{equation}
Applying the numerical traces above to Eq. \eqref{eq:conserv_cond1} and \eqref{eq:conserv_cond2},  we obtain the modified energy and Hamiltonian equations:
\begin{subequations}\label{EH}
    \begin{align}
        & \tau_{pu} \sum_{i=1}^{N} \jmp{u_h}^2(x_i)+\tau_{pq} \sum_{i=1}^{N} \jmp{u_h}\jmp{q_h}(x_i)=\sum_{i=1}^{N} \Big( \jmp{V(u_h)} - \{\Pi f(u_h)\}\jmp{u_h} \Big)(x_i), \label{eq:new_E_conserv}\\
     &\tau_{pu}  \sum_{i=1}^N \jmp{p_h}\jmp{u_h}(x_i)+\tau_{pq}  \sum_{i=1}^N \jmp{p_h}\jmp{q_h}(x_i)  = 0.\label{eq:new_H_conserv}
    \end{align}
\end{subequations}
The resulting DG formulation seeks $(u_h, q_h, p_h, \tau_{pu}, \tau_{pq})\in \left[W^k_h\right]^3 \times \mathbb{R}\times \mathbb{R}$ satisfying \eqref{DEqs} and \eqref{EH}, with numerical traces defined in \eqref{NTU}.

Importantly, the new conservation constraints in \eqref{eq:new_E_conserv}-\eqref{eq:new_H_conserv} no longer involve time-derivative-of-jump terms. This removal eliminates the main obstruction in the previous formulation and allows the use of higher-order time-stepping methods while maintaining conservation of mass, energy, and Hamiltonian. 
{This paper adopts} the following fourth-order implicit Runge--Kutta (IRK4) scheme:
\begin{equation}\label{eq:IRK4}
	u_h^{n+1} = u_h^n + \sqrt{3}\,\bigl(u_h^{n,2} - u_h^{n,1}\bigr),
\end{equation}
where the stage values \(u_h^{n,1}\) and \(u_h^{n,2}\) are obtained as the DG
solution of the coupled system
	\begin{align*}
		\frac{u^{n,1} - u^n}{\Delta t_n}
		&+ a_{11}\!\left(\epsilon u_{xxx}^{n,1} + \partial_x f(u^{n,1})\right)
		+ a_{12}\!\left(\epsilon u_{xxx}^{n,2} + \partial_x f(u^{n,2})\right) = g(x,t_n+c_1 \Delta t),\\
		\frac{u^{n,2} - u^n}{\Delta t_n}
		&+ a_{21}\!\left(\epsilon u_{xxx}^{n,1} + \partial_x f(u^{n,1})\right)
		+ a_{22}\!\left(\epsilon u_{xxx}^{n,2} + \partial_x f(u^{n,2})\right) = g(x,t_n+c_2 \Delta t),
	\end{align*}
with coefficients
$a_{11}=a_{22}=\frac14$, $a_{12}=\frac14-\frac{\sqrt{3}}{6}$,  $a_{21}=\frac14+\frac{\sqrt{3}}{6}$, $c_1=\frac12-\frac{\sqrt{3}}{6}$, and $c_2=\frac12+\frac{\sqrt{3}}{6}$.
This higher-order time discretization is naturally compatible with the conservative DG formulation due to the absence of temporal coupling in the conservation constraints. Thanks to the new formulation, once the intermediate stage solutions are obtained from the global nonlinear system, the updates for $u_h^{n+1}, q_h^{n+1}, p_h^{n+1}$, and the penalty parameters can be performed sequentially, with only linear solves involved. In contrast, the method in \cite{ChenDongPereira2022} required an additional nonlinear solve to get the update $p_h^{n+1}$ and the penalty parameters. Implementation details, including the matrix equations and a flowchart of the algorithm, are provided in the Appendix \ref{sec:implementation_kdv}.

\section{Conservative DG method for the HS-KdV System}\label{sec:hs-kdv}
In this section, we extend the conservative DG framework developed for the gKdV equation to the coupled HS–KdV system \eqref{eq:hskdv}. The coupling of the nonlinear equations and the three conservation constraints require a careful design of numerical fluxes in the DG formulation. Nevertheless, the overall construction follows the same principle: conservation properties are incorporated into the DG framework and penalty parameters are treated as new unknowns in appropriately chosen numerical traces. In what follows, we first present the definition of the conservative DG method for the HS–KdV system. Next, we derive sufficient conditions on the numerical traces under which the resulting DG methods are guaranteed to preserve the three invariants. Our choice of numerical traces satisfies these conditions and therefore ensures the desired conservation properties.

\subsection{The DG Method For HS-KdV}

To construct the DG method, we rewrite the coupled HS-KdV system \eqref{eq:hskdv} as a system of first–order equations: 
\begin{equation}\label{CKdVWeak}
	\begin{alignedat}{3}
	q - u_x &= 0, &\qquad p - q_x - 3u^2 &= 0, &\qquad u_t - a p_x - b (v^2)_x &= 0,\\  
	w - v_x &= 0, &\qquad r - w_x        &= 0, &\qquad  v_t + r_x + 3u w       & = 0,
\end{alignedat}
\end{equation}
with the initial conditions \(u(x,0)=u_0(x)\) and \(v(x,0)=v_0(x)\) as well as the periodic boundary conditions. 

We discretize \eqref{CKdVWeak} by seeking \((u_h,q_h,p_h,w_h,r_h,v_h)\in \left[W_h^k\right]^6\) that satisfies the following equations:
\begin{subequations}\label{DCKdV}
\begin{align}
    (q_h,\alpha) + (u_h,\alpha_x)
    - \langle \hat{u}_h, \alpha n \rangle &= 0, \label{4.3a}\\
    (p_h,\beta) + (q_h,\beta_x)
    - \langle \hat{q}_h, \beta n \rangle - (3u_h^2,\beta) &= 0, \label{4.3b}\\
    (u_{ht},\gamma) + (a p_h + b v_h^2,\gamma_x)
    - \langle a\hat{p}_h + b\widetilde{v_h^2}, \gamma n \rangle &= 0, \label{4.3c}\\
    (w_h,\xi) + (v_h,\xi_x)
    - \langle \hat{v}_h, \xi n \rangle &= 0, \label{4.3d}\\
    (r_h,\psi) + (w_h,\psi_x)
    - \langle \hat{w}_h, \psi n \rangle &= 0, \label{4.3e}\\
    (v_{ht},\phi) - (r_h,\phi_x)
    + \langle \hat{r}_h, \phi n \rangle + (3u_h w_h,\phi) &= 0, \label{4.3f}
\end{align}
\end{subequations}
for any $\alpha, \beta, \gamma, \xi, \psi, \phi \in W_h^k$, 
where
\begin{equation}
\begin{aligned}
\hat{u}_h &= \{u_h\}, & \hat{q}_h &= \{q_h\}, & \hat{p}_h &= \{p_h\} + \tau_{pu}\jmp{u_h} + \tau_{pv}\jmp{v_h},\\
 \hat{v}_h &= \{v_h\},  & \hat{w}_h &= \{w_h\}, & \hat{r}_h &= \{r_h\}, \qquad \widetilde{v_h^2} = \{\Pi v_h^2\}.
\end{aligned}
\label{4.4}
\end{equation}
Here, $\Pi$ is the $L^2$-orthogonal projection onto $W^k_h$, and the penalty parameters \(\tau_{pu}\) and \(\tau_{pv}\) are constants and are treated as additional new unknowns. This allows us to include the following two constraints in the DG formulation to enforce the conservation of energy and Hamiltonian:
\begin{subequations}\label{4.5}
\begin{align}
        \sum_{i=1}^N \Big[& \,a\tau_{pu}\jmp{u_h}\jmp{u_h} + a\tau_{pv}\jmp{u_h}\jmp{v_h}  -a \jmp{V(u_h)} + a\jmp{u_h}\{\Pi f(u_h)\} - b\Theta(u_h, v_h, v_h) \Big](x_i) = 0
        \\
        \begin{split}
        \sum_{i=1}^N \Big[ &\,\left(a\jmp{p_h} + b\jmp{\Pi v_h^2}\right)\left(a\tau_{pu}\jmp{u_h} + a\tau_{pv}\jmp{v_h}\right)
        + a\jmp{p_h}\left(\tau_{pu}\jmp{u_h} + \tau_{pv}\jmp{v_h}\right) \\
        &\quad+ 2b\Theta(q_h,w_h,v_h) - 2b\Theta(r_h,u_h,v_h) - 2b\Theta(u_h,w_h,w_h) - \Theta(p_h,v_h,v_h) \Big](x_i) \\
        &+ 6b\left( (u_h^2, \Pi (w_hv_h))-(u_h w_h, \Pi(u_h v_h))\right)=0,
        \end{split}
 \end{align}
 \end{subequations}
where  \(f(u)=3u^2\), \(V(\cdot)\) is an antiderivative of \(f(\cdot)\), and
\begin{equation}\label{eq:theta}
    \Theta(\phi_1,\phi_2,\phi_3) = \jmp{\phi_1}\{\phi_2\phi_3\}
    + \jmp{\phi_2\phi_3}\{\phi_1\}
    -\Big(\jmp{\phi_1}\{\Pi(\phi_2\phi_3)\}+\jmp{\phi_2}\{\Pi(\phi_1\phi_3)\}+\jmp{\phi_3}\{\Pi(\phi_1\phi_2)\}\Big).
\end{equation} 

Our conservative DG method is to find \((u_h,q_h,p_h,w_h,r_h,v_h,\tau_{pu}, \tau_{qu}) \in \left[W_h^k\right]^6\times\mathbb{R}\times \mathbb{R}\) that satisfy \eqref{DCKdV}--\eqref{4.5}.

\subsection{Conservation Properties} 

In this subsection, we investigate the conservation properties of the scheme introduced in the previous subsection. We begin by establishing, in the following theorem, general conditions on the numerical traces 
 under which a DG method satisfying \eqref{DCKdV} preserves the mass, \(L^2\)-energy, and Hamiltonian. The condition for energy conservation was previously derived in \cite{Pereira2023}; for completeness, we include the proof here. Building on this, we derive an additional condition that ensures Hamiltonian conservation, leading to a framework for simultaneous preservation of all three invariants. These conditions are then used to prove the conservation properties of the DG scheme defined by \eqref{DCKdV}--\eqref{4.5}.

\begin{theorem}\label{lemma:4.1}
Suppose that \((u_h,q_h,p_h,w_h,r_h,v_h)\) satisfy \eqref{DCKdV}.
\begin{enumerate}
\item[(i)] If \(\widehat{p}_h\) and \(\widetilde{v_h^2}\) are single-valued, then we have
\[
\frac{d}{dt}\int_{\mathcal{T}_h} u_h dx = 0 \qquad \mathrm{(mass~conservation).}
\]
\end{enumerate}
\item[(ii)] If all numerical traces in \eqref{DCKdV} are single-valued and satisfy the condition
\begin{equation}
\begin{aligned}
0 
 = \sum_{i=1}^N \Bigg[&
2a \Big( \jmp{u_h}(\hat{p}_h - \{p_h\}) + \jmp{p_h}(\hat{u}_h - \{u_h\}) \Big) - \frac{4}{3}b \Big( \jmp{v_h}(\hat{r}_h - \{r_h\}) + \jmp{r_h}(\hat{v}_h - \{v_h\}) \Big)\\
&- 2a  \jmp{q_h}(\hat{q}_h - \{q_h\}) 
+ \frac{4}{3}b \jmp{w_h}(\hat{w}_h - \{w_h\}) + 2b \Big( -\jmp{v_h^2}\{u_h\} + \jmp{u_h}(\widetilde{v_h^2} - \{v_h^2\}) \Big)\\
& - 2a\jmp{V(u_h)} 
+2a \Big( \jmp{u_h}\{\Pi f(u_h)\} - \jmp{\Pi f(u_h)}(\hat{u}_h - \{u_h\}) \Big)\\
& + 4b \Big( \jmp{v_h}\{\Pi u_h v_h\} - \jmp{\Pi u_h v_h}(\hat{v}_h - \{v_h\}) \Big) 
\Bigg](x_i).
\end{aligned}
\label{4.6}
\end{equation}
where $f(u)=3u^2$ and $V(u)=u^3$,
then we have
\begin{equation}\label{4.7}
   \frac{d}{dt}\int_{\mathcal{T}_h} \left(u_h^2 + \frac{2}{3} b v_h^2\right) dx = 0 \qquad \mathrm{(energy~ conservation).}
\end{equation}

\item[(iii)] If all numerical traces in \eqref{DCKdV} are single-valued and satisfy the condition
\begin{equation}
\begin{aligned}
0 = \sum_{i=1}^N \Bigg[ &(1+a)\left( - \jmp{q_h}\left(\hat{u}_{ht} - \{u_{ht}\}\right) - \jmp{u_{ht}}\left(\widehat{q}_h - \{q_h\}\right)
\right)+ 2b \left( -\jmp{v_{ht}}(\hat{w}_h - \{w_h\}) - \jmp{w_h}(\hat{v}_{ht} - \{v_{ht}\}) \right)\\
& + \left( a\jmp{p_h} + b\jmp{\Pi v_h^2} \right) \left(
a\left(\widehat{p}_h - \{p_h\}\right) + b\left(\widetilde{v_h^2} - \{\Pi v_h^2\}\right) \right)  
 + a\left( \jmp{p_h}(\hat{p}_h - \{p_h\}) \right) \\
& - 2b \left( -\jmp{v_h}\{\Pi v_hp_h\}+ \jmp{\Pi v_hp_h}(\hat{v}_h - \{v_h\}) \right)  + b\left( -\jmp{v_h^2}\{p_h\} + \jmp{p_h}(\widetilde{v_h^2} - \{v_h^2\}) \right)\\
 & -2b\left( \jmp{r_h}(\hat{r}_h - \{r_h\}) \right)   + 2b\tilde{\Theta}(q_h,w_h,v_h) - 2b\tilde{\Theta}(r_h,u_h,v_h) - 2b\tilde{\Theta}(u_h,w_h,w_h) \Bigg](x_i)+\mathcal{R}_{\Pi},
\end{aligned} \label{4.8}
\end{equation}
\vspace{-3ex}
where 
\begin{align*}
\mathcal{R}_{\Pi}=& 6b\Big( \left(u_h^2, \Pi (w_hv_h)\right)-\left(u_h w_h, \Pi(u_h v_h)\right)\Big),\\
\tilde{\Theta}(\phi_i,\phi_j,\phi_k)
=& \jmp{\phi_i}\{\phi_j\phi_k\}
+ \jmp{\phi_j\phi_k}\{\phi_i\}
- \sum_{\mathrm{cycl}}\Big[\jmp{\phi_i}\,\{\Pi (\phi_j \phi_k)\}-(\hat{\phi_i}-\{\phi_i\})\,\jmp{\Pi (\phi_j \phi_k)}\Big],
\end{align*}
where $
\sum_{\mathrm{cycl}} f(i,j,k)
= f(i,j,k) + f(j,k,i) + f(k,i,j)
$
denotes the sum over cyclic permutations of $(i,j,k)$,
then we have
\begin{equation}\label{4.9}
    \frac{d}{dt}\int_{\mathcal{T}_h}\left((1+a)(u_h^3 - \frac{1}{2}q_h^2) + b(u_hv_h^2 - w_h^2)\right) dx=0 \qquad \mathrm{(Hamiltonian~  conservation).}
\end{equation}
\end{theorem}

\begin{proof}
\renewcommand{\qedsymbol}{}
\begin{enumerate}
   \item[(i)] Mass conservation follows by taking \(\gamma = 1\) in \eqref{4.3c} and using the single-valuedness of \(\widehat{p}_h\) and \(\widetilde{v_h^2}\).
   \item[(ii)] Next, we establish the conservation of energy. Taking the test functions
   \[
   \alpha = 2a p_h-6a\Pi(u_h^2), \; \beta = -2a q_h, \; \gamma = 2u_h, \;\xi = -\frac{4}{3}b\,r_h-4b \Pi(u_h v_h), \; \psi = \frac{4}{3}b\,w_h,  \; \phi = \frac{4}{3}b\,v_h
   \]
   in \eqref{4.3a}-\eqref{4.3f}, respectively,  adding these equations together, and using integration by parts, we get
\begin{align*}
\frac{d}{dt}\int_{\mathcal{T}_h} \left( u_h^2 + \frac{2}{3}b\, v_h^2 \right) dx
&+ 2a\langle p_h,u_h n\rangle
- 2a\langle \hat{u}_h,p_h n\rangle
- 2a\langle \hat{p}_h,u_h n\rangle
- a\langle q_h,q_h n\rangle
+ 2a\langle \hat{q}_h,q_h n\rangle\\
&- \frac{4}{3}b\langle v_h,r_h n\rangle
+ \frac{4}{3}b\langle \hat{v}_h,r_h n\rangle
+ \frac{4}{3}b\langle \hat{r}_h,v_h n\rangle
+ \frac{2}{3}b\langle w_h,w_h n\rangle
- \frac{4}{3}b\langle \hat{w}_h,w_h\rangle\\
&-6a\langle u_h, (\Pi u_h^2)n \rangle +6a(u_{hx}, \Pi(u_h v_h))+6a\langle\widehat{u}_h, \Pi(u_h^2)n\rangle\\
&+ 2b\langle v_h^2,u_h n\rangle
- 2b\langle \widetilde{v_h^2},u_hn\rangle 
-4b\langle v_h, \Pi(u_hv_h) n\rangle +4b\langle\widehat{v}_h, \Pi(u_hv_h)n\rangle
= 0.
\end{align*}
Since all numerical traces are single-valued, the equation above can be rewritten as 
\[
\begin{aligned}
\frac{d}{dt}\int_{\mathcal{T}_h} \left( u_h^2 + \frac{2}{3}b\, v_h^2 \right) dx
= & -2a\langle \widehat{p}_h-p_h,(\widehat{u}_h-u_h) n\rangle
+ a\langle (\widehat{q}_h-q_h)^2, n\rangle
+ \frac{4}{3}b\langle \widehat{v}_h- v_h,(\widehat{r}_h-r_h)n\rangle\\
&- \frac{2}{3}b\langle (\widehat{w_h}-w_h)^2, n\rangle
-6a\langle\widehat{u}_h -u_h, \Pi(u_h^2)n\rangle
+ 2b\langle \widetilde{v_h^2}-v_h^2,u_h n\rangle  \\
&-4b\langle\widehat{v}_h -v_h, \Pi(u_hv_h)n\rangle
-6a(u_{hx}, (\Pi u_h^2)).
\end{aligned}
\]
Letting \(V(u)=u^3\) and using integration by parts, we can rewrite the last term on the right hand side as
\[
	6a\left(u_{hx}, \Pi (u_h^2)\right)=  6a(u_{hx}, u_h^2) =2a\langle u_h^3, n\rangle =2a \langle V(u_h), n\rangle.
\]
Finally, using the identity
\begin{equation}\label{4.17}
\langle \rho, v n \rangle
= \sum_{i=1}^{N} \left( \jmp{\rho}\{v\} + \jmp{v}\{\rho\} \right)(x_i) \quad \text{ for any }\rho, v \in W_h^k,
\end{equation}
we get that  the energy-conservation equation \eqref{4.7} holds if the condition \eqref{4.6} is satisfied. This completes the proof of the discrete energy conservation property.

\item[(iii)] To prove the Hamiltonian conservation properties in \eqref{4.9}, we first differentiate \eqref{4.3a} and \eqref{4.3d} with respect to $t$ and perform integration by parts, which yields
\begin{subequations}
\begin{align}
(q_{ht}, \alpha) - (u_{htx}, \alpha)
- \langle \widehat{u}_{ht} - u_{ht}, \alpha \rangle &= 0,
\label{4.19a}\\
(w_{ht}, \xi) - (v_{htx}, \xi)
- \langle \widehat{v}_{ht} - v_{ht}, \xi \rangle &= 0.
\label{4.19b}
\end{align}
\end{subequations}
Next, in \eqref{4.19a}, \eqref{4.19b}, \eqref{4.3b}, \eqref{4.3c},  \eqref{4.3e}, and \eqref{4.3f}, we take
$$\alpha=(1+a)q_h, \;  \xi=-2b w_h,\; \beta=(1+a)u_{ht}, \;\gamma=-(1+a)p_h - b\Pi v_h^2, \;\psi=-2b v_{ht}, \; \phi=2b\,(\Pi(u_h v_h)+r_h)$$
respectively,  sum the resulting identities,  and apply integration by parts to obtain
\begin{equation}\label{eq:with_S}
\begin{aligned}
&(1+a)(3u_h^2,u_{ht})
-(1+a)(q_{ht},q_h)
+ b(u_{ht}, v_h^2)
+ 2b(v_{ht}, u_h v_h)
-2b(w_{ht}, w_h)\\
= & (1+a)\langle \widehat{q}_h - q_h,\; (\widehat{u}_{ht}-u_{ht}) \, n\rangle  -\frac{1}{2}\langle \bigl(a (\hat{p}_h - p_h) + b(\widetilde{v_h^2} - \Pi v_h^2)\bigr)^2, n \rangle - \frac{a}{2}\langle (\widehat{p}_h - p_h)^2,\, n\rangle \\
&+ 2b\langle \widehat{v}_{ht}-v_{ht},\; (\widehat{w}_h-w_h)\, n\rangle 
+ b\langle (\widehat{r}_h-r_h)^2,\,n\rangle
+b\langle \widetilde{v^2}-v_h^2, p_h n\rangle + S,
\end{aligned}
\end{equation}
where
 \begin{align}\label{eq:S}
 S & =  2b(v_h v_{hx},p_h)
+ 2b\left(r_h,(\Pi (u_h v_h)_x\right) - 2b\langle \widehat{r}_h,\Pi (u_h v_h) n\rangle - 6b\left(u_h w_h,\Pi (u_h v_h) + r_h\right),
 \end{align}
 In deriving \eqref{eq:with_S}, we have used the assumption that all numerical traces are single-valued across element interfaces.
We then apply the identity \eqref{4.17} to rewrite the boundary integral terms in \eqref{eq:with_S}, obtaining
\begin{equation*}
\begin{aligned}
 &\frac{d}{dt}\int_{\mathcal{T}_h} \left((1+a)(u_h^3 - \frac{1}{2}q_h^2) + b(u_hv_h^2 - w_h^2)\right) dx \\
 = \sum_{i=1}^N \Bigg[ &(1+a)\left( - \jmp{q_h}\left(\hat{u}_{ht} - \{u_{ht}\}\right) - \jmp{u_{ht}}\left(\widehat{q}_h - \{q_h\}\right)
\right)+ 2b \left( -\jmp{v_{ht}}(\hat{w}_h - \{w_h\}) - \jmp{w_h}(\hat{v}_{ht} - \{v_{ht}\}) \right)\\
& + \left( a\jmp{p_h} + b\jmp{\Pi v_h^2} \right) \left(
a\left(\widehat{p}_h - \{p_h\}\right) + b\left(\widetilde{v_h^2} - \{\Pi v_h^2\}\right) \right)  
 + a\left( \jmp{p_h}(\hat{p}_h - \{p_h\}) \right)\\
& -2b\left( \jmp{r_h}(\hat{r}_h - \{r_h\}) \right) + b\left( -\jmp{v_h^2}\{p_h\} + \jmp{p_h}(\widetilde{v_h^2} - \{v_h^2\}) \right) \Bigg](x_i)+S.
\end{aligned} \label{4.8}
\end{equation*}
 To complete the proof of (iii) of Theorem \eqref{lemma:4.1}, we just need to use the following Lemma, whose proof is provided in Appendix \ref{appendix:proof_lemma}.
\begin{lemma}\label{lemma:nonlinear-coupling}
Let $(u_h,q_h,p_h,v_h,w_h,r_h)\in [W_h^k]^6$ satisfy the DG formulation
\eqref{4.3a}--\eqref{4.3f}. Then $S$, defined in \eqref{eq:S}, can be rewritten as
\begin{equation*} 
\begin{aligned}
S = \sum_{i=1}^N \Bigg[
& - 2b \left( -\jmp{v_h}\{\Pi v_hp_h\}+ \jmp{\Pi v_hp_h}(\hat{v}_h - \{v_h\}) \right)  \\
  & + 2b\tilde{\Theta}(q_h,w_h,v_h) - 2b\tilde{\Theta}(r_h,u_h,v_h) - 2b\tilde{\Theta}(u_h,w_h,w_h) \Bigg](x_i)+ \mathcal{R}_{\Pi},
\end{aligned}
\end{equation*}
where $\tilde{\Theta}(\cdot, \cdot, \cdot)$ and $\mathcal{R}_{\Pi}$ are defined as in part (iii) of Theorem \ref{lemma:4.1}.
\end{lemma}

\end{enumerate}
\end{proof}

\begin{remark}
	We emphasize that Theorem \ref{lemma:4.1} provides a general framework for achieving the simultaneous conservation of mass, energy, and Hamiltonian. In particular, any choice of numerical traces \(\hat{u}_h\), \(\hat{q}_h\), \(\hat{p}_h\), \(\hat{w}_h\), \(\hat{r}_h\), \(\hat{v}_h\), and \(\widetilde{v_h^2}\) satisfying the conditions \eqref{4.6} and \eqref{4.8} yields a fully conservative scheme. 
	
This flexibility highlights that conservation is determined by structural constraints rather than a specific choice of numerical traces, and it provides a foundation for designing alternative conservative DG formulations.
\end{remark}

\begin{remark}
	To avoid the appearance of time derivatives of average or jump terms in the Hamiltonian conservation constraint \eqref{4.8}, the numerical traces must satisfy the following restriction:
    \[ 
     (1+a)\left(\jmp{q_h}(\hat{u}_{ht} -\{u_{ht}\}) + \jmp{u_{ht}}(\widehat{q}_h - \{q_h\})\right)+2b\left(\jmp{v_{ht}}(\hat{w}_h - \{w_h\}) + \jmp{w_h}(\hat{v}_{ht} - \{v_{ht}\})\right)=0. 
    \]
    A simple choice is
	\begin{equation}\label{eq:trace_restriction}
		\widehat{u}_h=\{u_h\}, \quad \widehat{q}_h=\{q_h\}, \quad \widehat{v}_h=\{v_h\}, \quad \widehat{w}_h=\{ w_h\}.
	\end{equation}
	For the remaining three numerical traces \(\widehat{p}_h\), \(\widehat{r}_h\), and \(\widetilde{v_h^2}\), there remains some flexibility in their design. For example, one may take either (i) two of them to be central fluxes and the remaining one with a penalty term, or (ii) one to be a central flux and the other two with penalty terms. 
    In this work, we adopt a specific choice that satisfies \eqref{eq:trace_restriction} to eliminate the time derivative of average and jump terms, thereby suitable for the use of higher-order time discretizations.
\end{remark}

\begin{remark}
The term
$\mathcal{R}_{\Pi}= 6b\Big( \left(u_h^2, \Pi (w_h v_h)\right)-\left(u_h w_h, \Pi(u_h v_h)\right)\Big)$
arises in the condition \eqref{4.8} for Hamiltonian conservation from the use of the $L^2$ projection in the derivation. It can be viewed as a projection-induced, higher-order consistency remainder. This term vanishes when $\Pi$ is exact on products, i.e. $\Pi(w_h v_h)=w_h v_h$ and $\Pi (u_h v_h)=u_h v_h$; for the DG polynomial spaces $W_h^k$ and the $L^2$ projection $\Pi$ used here, this occurs only for $k=0$. 
\end{remark}

Applying Theorem \ref{lemma:4.1} to the scheme defined by \eqref{DCKdV}--\eqref{4.5}, we immediately get the following theorem.

\begin{theorem} 
For \((u_h,q_h,p_h,w_h,r_h,v_h)\) satisfying \eqref{DCKdV}, and with numerical traces defined by \eqref{4.4}--\eqref{4.5}, the mass, \(L^2\)-energy, and Hamiltonian conservation properties stated in Theorem~\textup{\ref{lemma:4.1}} hold.
\end{theorem}

\section{Numerical Experiments}\label{sec:numerics}
We carry out numerical experiments to test the accuracy and conservation properties of the conservative DG methods for the gKdV equation and the HS-KdV system, and the results are presented in Subsections \ref{sec:numerics_gkdv} and \ref{sec:numerics_hs-kdv}, respectively. For all test problems, we compute the \(L^2\)-errors and the corresponding convergence rates using the IRK4 scheme, and we also verify the conservation of the discrete energy and Hamiltonian associated with the DG solutions. 

\subsection{Numerical Results for the gKdV Equation}\label{sec:numerics_gkdv}

In this subsection, we present three numerical experiments to examine the convergence behavior and conservation properties of the proposed DG method for the gKdV equation. In the first test problem, we consider a third-order linear equation with \(f(u)=u\). In the second test problem, we consider a third-order nonlinear equation with different parameters \(\varepsilon = 1,\;0.1,\;0.01\), where the solutions are periodic sine waves on the computational domain. In the third test problem, we solve the classical KdV equation with a cnoidal wave solution and compare the numerical approximation with the exact solution.

\subsubsection{Experiment 1: Linear Third-Order Equation}

In this experiment, we consider the third-order linear equation
$u_t + \epsilon u_{xxx} + (f(u))_x = 0$, 
where \(\epsilon = 1\) and \(f(u) = u\). Periodic boundary conditions are imposed on the domain
\(\Omega = [0,4\pi]\), and the initial condition is prescribed as \(u_0(x) = \sin\!\left(\tfrac{1}{2}x\right)\).
The corresponding exact solution is $u(x,t) = \sin\!\left(\tfrac{1}{2}x - \tfrac{3}{8}t\right)$.

We first study the convergence behavior of the DG method for this linear problem. Polynomial approximations of degree \(k = 0,1,2, 3, 4\) are employed. The mesh size is defined by
$
h = \frac{4\pi}{N}, \text{where } N = 2^{\,l}, \; l = 3,\ldots,7,
$
and the time step is chosen as \(\Delta t=0.2h\) for \(k=0,1,2\) and \(\Delta t=4h^2\) for \(k=3,4\).
The \(L^2\)-errors and corresponding convergence orders of the numerical solutions are reported in Table \ref{tab:2}. The results indicate that the numerical approximation to the primary variable \(u\) achieves optimal convergence rates for even polynomial degrees, while suboptimal rates are observed for odd degrees, which is typical for energy-conserving DG methods.
The approximation to the auxiliary variable \(q\) exhibits similar convergence behaviors to that of $u$. 
The approximation to \(p\) has comparatively weaker convergence, with optimal convergence order essentially restricted to $k=0$.

We remark that the use of a higher-order time integration scheme, such as IRK4, enables more effective use of higher-degree polynomial approximations and allows the spatial discretization to realize its full potential. It also permits larger time steps, making the method suitable for long-time simulations.

We next examine the conservation properties of the fully discrete scheme. Using IRK4 with polynomial degree \(k=2\) on \(32\) uniform elements, the solution is evolved to the final time \(T=50\). As shown in Fig. \ref{fig:2}, both the discrete energy and the discrete Hamiltonian remain constant throughout the simulation, with conservation errors that are extremely small.  

\begin{table}[H]
	\centering
	\caption{Numerical Experiment 1 (third-order linear equation): \(L^2\)-errors and convergence orders. }
	\label{tab:2}
	\renewcommand{\arraystretch}{1}
	\begin{tabular}{cc cc cc cc}
		\toprule
		$k$ & $N$
		& \multicolumn{2}{c}{$u_h$}
		& \multicolumn{2}{c}{$q_h$}
		& \multicolumn{2}{c}{$p_h$} \\
		\cmidrule(lr){3-4}
		\cmidrule(lr){5-6}
		\cmidrule(lr){7-8}
		&  & $L_2$ Error & Order & $L_2$ Error & Order & $L_2$ Error & Order \\
		\midrule
		
		\multirow{5}{*}{0}
		& 8   & $5.65\mathrm{e}{-1}$ & --
		& $3.07\mathrm{e}{-1}$ & --
		& $4.39\mathrm{e}{-1}$ & -- \\
		& 16  & $2.95\mathrm{e}{-1}$ & 0.94
		& $1.50\mathrm{e}{-1}$ & 1.03
		& $2.23\mathrm{e}{-1}$ & 0.97 \\
		& 32  & $1.41\mathrm{e}{-1}$ & 1.06
		& $7.14\mathrm{e}{-2}$ & 1.08
		& $1.06\mathrm{e}{-1}$ & 1.07 \\
		& 64  & $7.17\mathrm{e}{-2}$ & 0.99
		& $3.59\mathrm{e}{-2}$ & 0.99
		& $5.38\mathrm{e}{-2}$ & 0.99 \\
		& 128 & $3.55\mathrm{e}{-2}$ & 1.01
		& $1.78\mathrm{e}{-2}$ & 1.01
		& $2.66\mathrm{e}{-2}$ & 1.01 \\
		\midrule
		
		\multirow{5}{*}{1}
		& 8   & $5.72\mathrm{e}{-2}$ & --
		& $2.55\mathrm{e}{-1}$ & --
		& $2.33\mathrm{e}{-1}$ & -- \\
		& 16  & $1.60\mathrm{e}{-2}$ & 1.83
		& $1.38\mathrm{e}{-1}$ & 0.89
		& $1.35\mathrm{e}{-1}$ & 0.79 \\
		& 32  & $6.30\mathrm{e}{-3}$ & 1.35
		& $7.04\mathrm{e}{-2}$ & 0.97
		& $6.99\mathrm{e}{-2}$ & 0.95 \\
		& 64  & $2.79\mathrm{e}{-3}$ & 1.17
		& $3.53\mathrm{e}{-2}$ & 0.99
		& $3.53\mathrm{e}{-2}$ & 0.99 \\
		& 128 & $1.31\mathrm{e}{-3}$ & 1.09
		& $1.77\mathrm{e}{-2}$ & 1.00
		& $1.77\mathrm{e}{-2}$ & 1.00 \\
		\midrule
		
		\multirow{4}{*}{2}
		& 8   & $4.75\mathrm{e}{-3}$ & --
		& $9.12\mathrm{e}{-3}$ & --
		& $3.66\mathrm{e}{-2}$ & -- \\
		& 16  & $4.89\mathrm{e}{-4}$ & 3.28
		& $9.56\mathrm{e}{-4}$ & 3.25
		& $9.46\mathrm{e}{-3}$ & 1.95 \\
		& 32  & $6.06\mathrm{e}{-5}$ & 3.01
		& $1.23\mathrm{e}{-4}$ & 2.96
		& $2.32\mathrm{e}{-3}$ & 2.02 \\
		& 64  & $7.47\mathrm{e}{-6}$ & 3.02
		& $5.23\mathrm{e}{-6}$ & 4.55
		& $1.37\mathrm{e}{-4}$ & 4.08 \\
		\midrule
		
		\multirow{3}{*}{3}
		& 8   & $5.91\mathrm{e}{-4}$ & --
		& $2.03\mathrm{e}{-3}$ & --
		& $6.59\mathrm{e}{-3}$ & -- \\
		& 16  & $1.25\mathrm{e}{-4}$ & 2.24
		& $2.34\mathrm{e}{-4}$ & 3.12
		& $1.37\mathrm{e}{-3}$ & 2.26 \\
		& 32  & $1.43\mathrm{e}{-5}$ & 3.12
		& $3.05\mathrm{e}{-5}$ & 2.94
		& $4.50\mathrm{e}{-4}$ & 1.61 \\
		\midrule
		
		\multirow{3}{*}{4}
		& 8   & $7.66\mathrm{e}{-6}$ & --
		& $3.64\mathrm{e}{-5}$ & --
		& $4.42\mathrm{e}{-4}$ & -- \\
		& 16  & $2.35\mathrm{e}{-7}$ & 5.03
		& $1.12\mathrm{e}{-6}$ & 5.02
		& $2.76\mathrm{e}{-5}$ & 4.00 \\
		& 32  & $7.31\mathrm{e}{-9}$ & 5.01
		& $3.50\mathrm{e}{-8}$ & 5.00
		& $1.72\mathrm{e}{-6}$ & 4.00 \\
		\bottomrule
	\end{tabular}
\end{table}

\begin{figure}[htbp]
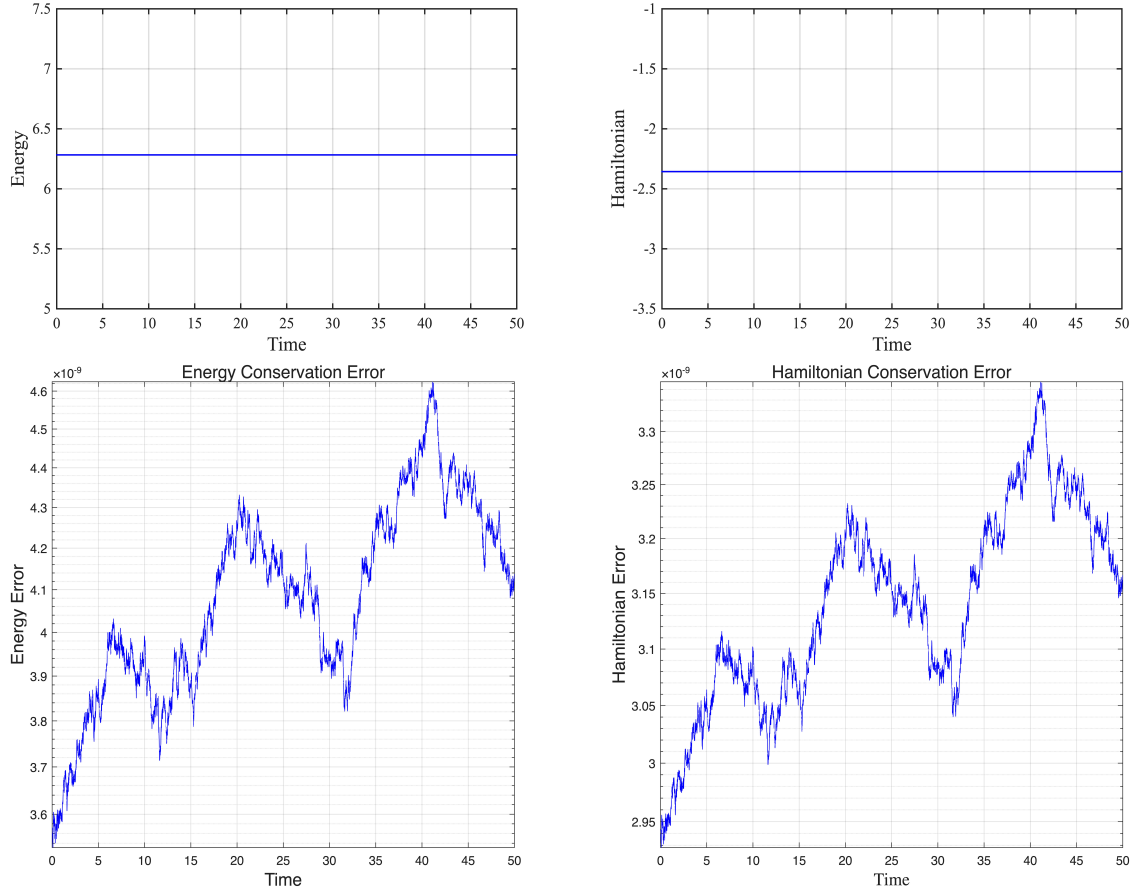

	\centering
	\includegraphics[width=0.9\textwidth]{./Graphics/P1K2.jpeg}
    
    \vspace{3pt}
    
	\includegraphics[width=0.9\textwidth]{./Graphics/P1K2EHErrorPlot.jpeg}
	\caption{Numerical Experiment 1 (third-order linear equation): Energy with error (left) and Hamiltonian with error (right) conservation.}
	\label{fig:2}
\end{figure}

\subsubsection{Experiment 2: Nonlinear Third-Order Equation}

In this numerical test, we consider the third–order nonlinear equation 
$u_t + \varepsilon u_{xxx} + (f(u))_x = g$,
subject to periodic boundary conditions on \(\Omega=[0,1]\) and the initial condition
\(u_0(x)=\sin(2\pi x)\). Here \(f(u)=\tfrac{u^2}{2}\), and the source term \(g\) is chosen such that the exact solution is
$u(x,t)=\sin(2\pi x+t)$.
We first investigate the convergence behavior of the DG method for different parameters 
\(\varepsilon = 1,\;0.1,\;0.01\) using polynomial degrees \(k=0,1,2, 3, 4\).
The mesh size is \(h=1/N\), where \(N=2^l,\; l=3,\ldots,7\), and the final time is \(T=0.1\). The time step is chosen as \(\Delta t=0.2h\) for \(k=0,1,2\) and \(\Delta t=4h^2\) for \(k=3,4\). The \(L^2\)-errors and convergence orders obtained with IRK4 are presented in Tables \ref{tab:combined_eps1}-\ref{tab:combined_eps001}. 

From Tables~\ref{tab:combined_eps1}--\ref{tab:combined_eps001}, consistent convergence trends are observed across all values of $\varepsilon$. The primary variable $u_h$ exhibits robust convergence behavior, achieving optimal rates for even polynomial degrees, while suboptimal or irregular convergence is observed for odd degrees. The auxiliary variable $q_h$ shows a similar convergence structure to $u_h$. The variable $p_h$ has optimal convergence order only for $k=0$.

Finally, we examine the conservation properties of the proposed DG scheme. The energy and Hamiltonian of the numerical solutions  obtained by IRK4 with polynomial degree
$k=2$ on 32 uniform elements over the time interval $t \in [0,50]$ are constant in time for all tested values of $\varepsilon$, similar to those in Fig. \ref{fig:2}. The corresponding conservation errors for both quantities are presented in Fig. \ref{fig:5}. The results demonstrate that the method preserves the energy and Hamiltonian to a high degree of accuracy. Although these invariants are proved to be conserved for the KdV equation when the source term satisfies $g \equiv 0$, the manufactured solution used here includes a nonzero source term while still preserving the corresponding invariants, providing a suitable test for the numerical scheme.

\begin{table}[H]
	\centering
	\caption{Numerical Experiment 2 (third-order nonlinear equation): \(L^2\)-errors and convergence orders for \(\varepsilon=1\).}
	\label{tab:combined_eps1}
	\renewcommand{\arraystretch}{1}
	\begin{tabular}{cc cc cc cc}
		\toprule
		$k$ & $N$
		& \multicolumn{2}{c}{$u_h$}
		& \multicolumn{2}{c}{$q_h$}
		& \multicolumn{2}{c}{$p_h$} \\
		\cmidrule(lr){3-4}
		\cmidrule(lr){5-6}
		\cmidrule(lr){7-8}
		&  & $L_2$ Error & Order & $L_2$ Error & Order & $L_2$ Error & Order \\
		\midrule
		\multirow{5}{*}{0}
		& 8   & $2.50\mathrm{e}{-1}$ & --
		& $1.20\mathrm{e}{+0}$ & --
		& $6.51\mathrm{e}{+0}$ & -- \\
		& 16  & $1.00\mathrm{e}{-1}$ & 1.32
		& $5.69\mathrm{e}{-1}$ & 1.09
		& $3.32\mathrm{e}{+0}$ & 0.97 \\
		& 32  & $4.50\mathrm{e}{-2}$ & 1.16
		& $2.70\mathrm{e}{-1}$ & 1.07
		& $1.64\mathrm{e}{+0}$ & 1.02 \\
		& 64  & $2.04\mathrm{e}{-2}$ & 1.14
		& $1.27\mathrm{e}{-1}$ & 1.09
		& $7.97\mathrm{e}{-1}$ & 1.04 \\
		& 128 & $1.00\mathrm{e}{-2}$ & 1.02
		& $6.31\mathrm{e}{-2}$ & 1.01
		& $3.96\mathrm{e}{-1}$ & 1.01 \\
		\midrule
		
		\multirow{5}{*}{1}
		& 8   & $1.33\mathrm{e}{-1}$ & --
		& $8.26\mathrm{e}{-1}$ & --
		& $1.55\mathrm{e}{+1}$ & -- \\
		& 16  & $2.72\mathrm{e}{-2}$ & 2.29
		& $4.98\mathrm{e}{-1}$ & 0.73
		& $6.46\mathrm{e}{+0}$ & 1.27 \\
		& 32  & $1.30\mathrm{e}{-2}$ & 1.06
		& $2.18\mathrm{e}{-1}$ & 1.19
		& $3.54\mathrm{e}{+0}$ & 0.87 \\
		& 64  & $3.63\mathrm{e}{-4}$ & 5.17
		& $1.25\mathrm{e}{-1}$ & 0.80
		& $1.57\mathrm{e}{+0}$ & 1.17 \\
		& 128 & $8.84\mathrm{e}{-5}$ & 2.04
		& $6.29\mathrm{e}{-2}$ & 1.00
		& $7.90\mathrm{e}{-1}$ & 1.00 \\
		\midrule
		
		\multirow{4}{*}{2}
		& 8   & $1.07\mathrm{e}{-3}$ & --
		& $2.85\mathrm{e}{-2}$ & --
		& $1.82\mathrm{e}{+0}$ & -- \\
		& 16  & $1.34\mathrm{e}{-4}$ & 2.99
		& $3.47\mathrm{e}{-3}$ & 3.04
		& $4.26\mathrm{e}{-1}$ & 2.10 \\
		& 32  & $1.68\mathrm{e}{-5}$ & 3.00
		& $4.30\mathrm{e}{-4}$ & 3.01
		& $1.04\mathrm{e}{-1}$ & 2.04 \\
		& 64  & $2.10\mathrm{e}{-6}$ & 3.00
		& $5.37\mathrm{e}{-5}$ & 3.00
		& $2.58\mathrm{e}{-2}$ & 2.01 \\
		\midrule
		
		\multirow{2}{*}{3}
		& 8   & $5.45\mathrm{e}{-5}$ & --
		& $7.33\mathrm{e}{-3}$ & --
		& $3.95\mathrm{e}{-1}$ & -- \\
		& 16  & $3.38\mathrm{e}{-6}$ & 4.01
		& $9.88\mathrm{e}{-4}$ & 2.89
		& $7.97\mathrm{e}{-2}$ & 2.31 \\
		\midrule
		
		\multirow{3}{*}{4}
		& 8   & $4.11\mathrm{e}{-6}$ & --
		& $1.20\mathrm{e}{-4}$ & --
		& $1.79\mathrm{e}{-2}$ & -- \\
		& 16  & $6.95\mathrm{e}{-7}$ & 2.56
		& $2.78\mathrm{e}{-6}$ & 5.43
		& $3.04\mathrm{e}{-4}$ & 5.89 \\
		& 32  & $1.26\mathrm{e}{-8}$ & 5.78
		& $1.13\mathrm{e}{-7}$ & 4.62
		& $6.97\mathrm{e}{-5}$ & 2.12 \\
		\bottomrule
	\end{tabular}
\end{table}

\begin{table}[H]
	\centering
	\caption{Numerical Experiment 2 (third-order nonlinear equation): \(L^2\)-errors and convergence orders for \(\varepsilon=0.1\).}
	\label{tab:combined_eps01}
	\renewcommand{\arraystretch}{1}
	\begin{tabular}{cc cc cc cc}
		\toprule
		$k$ & $N$
		& \multicolumn{2}{c}{$u_h$}
		& \multicolumn{2}{c}{$q_h$}
		& \multicolumn{2}{c}{$p_h$} \\
		\cmidrule(lr){3-4}
		\cmidrule(lr){5-6}
		\cmidrule(lr){7-8}
		&  & $L_2$ Error & Order & $L_2$ Error & Order & $L_2$ Error & Order \\
		\midrule
		
		\multirow{5}{*}{0}
		& 8   & $2.27\mathrm{e}{-1}$ & --
		& $1.16\mathrm{e}{+0}$ & --
		& $6.93\mathrm{e}{-1}$ & -- \\
		& 16  & $1.03\mathrm{e}{-1}$ & 1.14
		& $5.86\mathrm{e}{-1}$ & 0.98
		& $3.48\mathrm{e}{-1}$ & 0.99 \\
		& 32  & $4.36\mathrm{e}{-2}$ & 1.24
		& $2.65\mathrm{e}{-1}$ & 1.15
		& $1.64\mathrm{e}{-1}$ & 1.09 \\
		& 64  & $2.05\mathrm{e}{-2}$ & 1.09
		& $1.28\mathrm{e}{-1}$ & 1.05
		& $8.03\mathrm{e}{-2}$ & 1.03 \\
		& 128 & $1.01\mathrm{e}{-2}$ & 1.02
		& $6.32\mathrm{e}{-2}$ & 1.01
		& $3.99\mathrm{e}{-2}$ & 1.01 \\
		\midrule
		
		\multirow{5}{*}{1}
		& 8   & $5.66\mathrm{e}{-2}$ & --
		& $1.03\mathrm{e}{+0}$ & --
		& $1.14\mathrm{e}{+0}$ & -- \\
		& 16  & $1.71\mathrm{e}{-2}$ & 1.73
		& $5.12\mathrm{e}{-1}$ & 1.01
		& $6.21\mathrm{e}{-1}$ & 0.88 \\
		& 32  & $4.47\mathrm{e}{-3}$ & 1.94
		& $2.53\mathrm{e}{-1}$ & 1.02
		& $3.15\mathrm{e}{-1}$ & 0.98 \\
		& 64  & $1.14\mathrm{e}{-3}$ & 1.97
		& $1.26\mathrm{e}{-1}$ & 1.00
		& $1.59\mathrm{e}{-1}$ & 0.99 \\
		\midrule
		
		\multirow{4}{*}{2}
		& 8   & $1.08\mathrm{e}{-3}$ & --
		& $2.86\mathrm{e}{-2}$ & --
		& $1.83\mathrm{e}{-1}$ & -- \\
		& 16  & $1.35\mathrm{e}{-4}$ & 3.00
		& $3.47\mathrm{e}{-3}$ & 3.04
		& $4.27\mathrm{e}{-2}$ & 2.10 \\
		& 32  & $1.72\mathrm{e}{-5}$ & 2.97
		& $4.33\mathrm{e}{-4}$ & 3.01
		& $1.03\mathrm{e}{-2}$ & 2.04 \\
		& 64  & $2.11\mathrm{e}{-6}$ & 3.03
		& $5.37\mathrm{e}{-5}$ & 3.01
		& $2.58\mathrm{e}{-3}$ & 2.00 \\
		\midrule
		
		\multirow{3}{*}{3}
		& 8   & $5.37\mathrm{e}{-5}$ & --
		& $7.34\mathrm{e}{-3}$ & --
		& $3.96\mathrm{e}{-2}$ & -- \\
		& 16  & $1.14\mathrm{e}{-5}$ & 2.24
		& $9.73\mathrm{e}{-4}$ & 2.92
		& $7.98\mathrm{e}{-3}$ & 2.31 \\
		& 32  & $1.25\mathrm{e}{-6}$ & 3.19
		& $1.17\mathrm{e}{-4}$ & 3.06
		& $1.70\mathrm{e}{-3}$ & 2.23 \\
		\midrule
		
		\multirow{3}{*}{4}
		& 8   & $2.15\mathrm{e}{-6}$ & --
		& $1.28\mathrm{e}{-4}$ & --
		& $1.95\mathrm{e}{-3}$ & -- \\
		& 16  & $6.58\mathrm{e}{-8}$ & 5.03
		& $3.98\mathrm{e}{-6}$ & 5.01
		& $1.23\mathrm{e}{-4}$ & 3.99 \\
		& 32  & $2.06\mathrm{e}{-9}$ & 5.00
		& $1.24\mathrm{e}{-7}$ & 5.00
		& $7.67\mathrm{e}{-6}$ & 4.00 \\
		\bottomrule
	\end{tabular}
\end{table}

\begin{table}[H]
	\centering
	\caption{Numerical Experiment 2 (third-order nonlinear equation): \(L^2\)-errors and convergence orders for \(\varepsilon=0.01\).}
	\label{tab:combined_eps001}
	\renewcommand{\arraystretch}{1}
	\begin{tabular}{cc cc cc cc}
		\toprule
		$k$ & $N$
		& \multicolumn{2}{c}{$u_h$}
		& \multicolumn{2}{c}{$q_h$}
		& \multicolumn{2}{c}{$p_h$} \\
		\cmidrule(lr){3-4}
		\cmidrule(lr){5-6}
		\cmidrule(lr){7-8}
		&  & $L_2$ Error & Order & $L_2$ Error & Order & $L_2$ Error & Order \\
		\midrule
		
		\multirow{5}{*}{0}
		& 8   & $1.72\mathrm{e}{-1}$ & --
		& $1.19\mathrm{e}{+0}$ & --
		& $1.28\mathrm{e}{-1}$ & -- \\
		& 16  & $8.22\mathrm{e}{-2}$ & 1.06
		& $5.54\mathrm{e}{-1}$ & 1.10
		& $5.82\mathrm{e}{-2}$ & 1.14 \\
		& 32  & $4.03\mathrm{e}{-2}$ & 1.03
		& $2.59\mathrm{e}{-1}$ & 1.09
		& $2.67\mathrm{e}{-2}$ & 1.13 \\
		& 64  & $2.01\mathrm{e}{-2}$ & 1.01
		& $1.27\mathrm{e}{-1}$ & 1.03
		& $1.29\mathrm{e}{-2}$ & 1.04 \\
		& 128 & $1.00\mathrm{e}{-2}$ & 1.00
		& $6.31\mathrm{e}{-2}$ & 1.01
		& $6.40\mathrm{e}{-3}$ & 1.01 \\
		\midrule
		
		\multirow{5}{*}{1}
		& 8   & $7.90\mathrm{e}{-2}$ & --
		& $5.08\mathrm{e}{-1}$ & --
		& $1.38\mathrm{e}{-1}$ & -- \\
		& 16  & $1.10\mathrm{e}{-2}$ & 2.85
		& $4.93\mathrm{e}{-1}$ & 0.05
		& $6.85\mathrm{e}{-2}$ & 1.01 \\
		& 32  & $1.32\mathrm{e}{-3}$ & 3.06
		& $2.51\mathrm{e}{-1}$ & 0.97
		& $3.17\mathrm{e}{-2}$ & 1.11 \\
		& 64  & $6.98\mathrm{e}{-4}$ & 0.92
		& $1.26\mathrm{e}{-1}$ & 0.99
		& $1.65\mathrm{e}{-2}$ & 0.94 \\
		& 128 & $3.19\mathrm{e}{-4}$ & 1.13
		& $6.35\mathrm{e}{-2}$ & 0.99
		& $8.41\mathrm{e}{-3}$ & 0.98 \\
		\midrule
		
		\multirow{4}{*}{2}
		& 8   & $1.41\mathrm{e}{-3}$ & --
		& $3.22\mathrm{e}{-2}$ & --
		& $1.81\mathrm{e}{-2}$ & -- \\
		& 16  & $1.35\mathrm{e}{-4}$ & 3.38
		& $3.48\mathrm{e}{-3}$ & 3.21
		& $4.28\mathrm{e}{-3}$ & 2.08 \\
		& 32  & $1.69\mathrm{e}{-5}$ & 3.00
		& $4.31\mathrm{e}{-4}$ & 3.01
		& $1.04\mathrm{e}{-3}$ & 2.04 \\
		& 64  & $2.11\mathrm{e}{-6}$ & 3.00
		& $5.44\mathrm{e}{-5}$ & 2.98
		& $2.62\mathrm{e}{-4}$ & 1.99 \\
		\midrule
		
		\multirow{3}{*}{3}
		& 8   & $1.08\mathrm{e}{-4}$ & --
		& $7.21\mathrm{e}{-3}$ & --
		& $3.97\mathrm{e}{-3}$ & -- \\
		& 16  & $5.12\mathrm{e}{-5}$ & 1.07
		& $1.03\mathrm{e}{-3}$ & 2.81
		& $7.63\mathrm{e}{-4}$ & 2.38 \\
		& 32  & $4.19\mathrm{e}{-6}$ & 3.61
		& $1.63\mathrm{e}{-4}$ & 2.65
		& $2.11\mathrm{e}{-4}$ & 1.85 \\
		\midrule
		
		\multirow{2}{*}{4}
		& 8   & $2.11\mathrm{e}{-6}$ & --
		& $1.28\mathrm{e}{-4}$ & --
		& $1.96\mathrm{e}{-4}$ & -- \\
		& 16  & $6.59\mathrm{e}{-8}$ & 5.00
		& $3.98\mathrm{e}{-6}$ & 5.01
		& $1.23\mathrm{e}{-5}$ & 3.99 \\
		\bottomrule
	\end{tabular}
\end{table}

\begin{figure}[ht]
	\centering
	\includegraphics[width=0.66\textwidth]{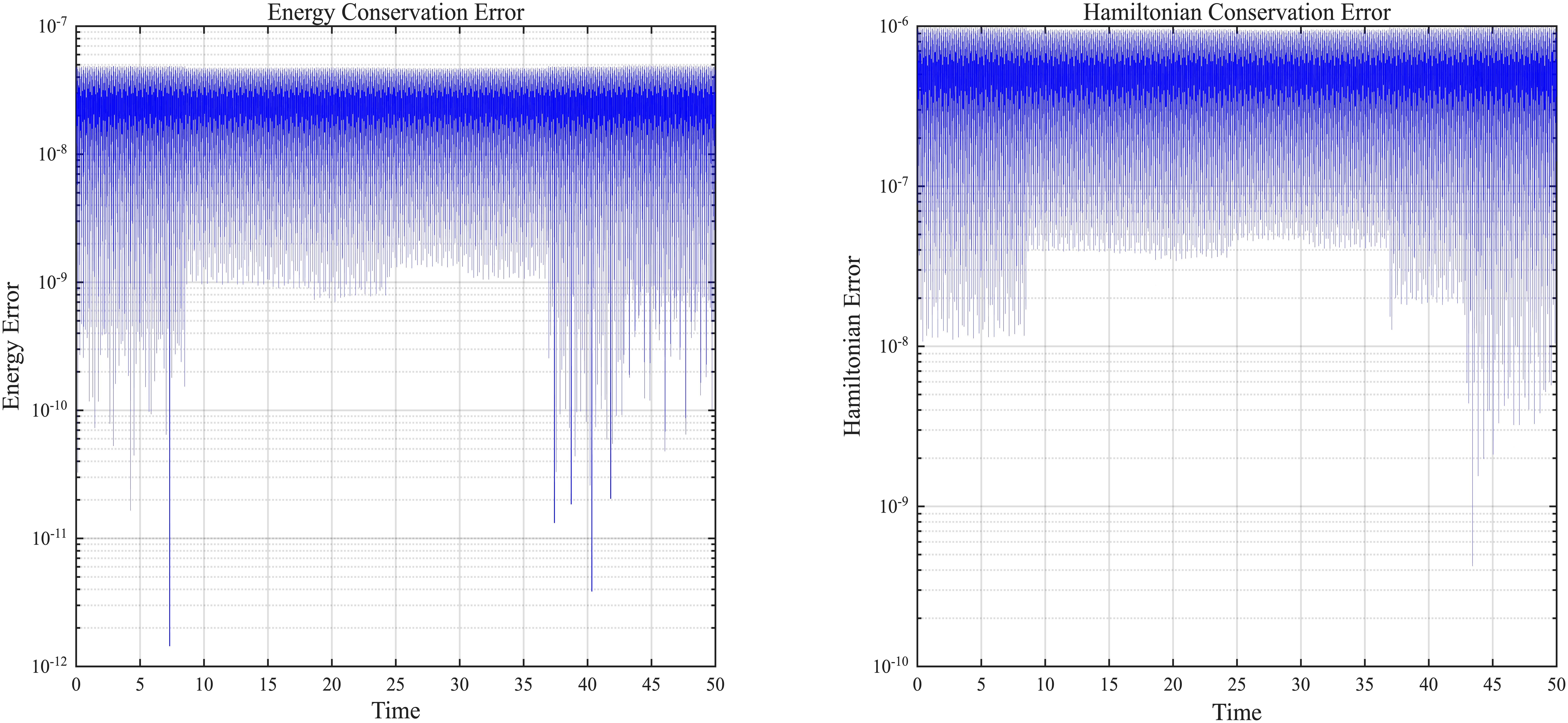}
	
	\vspace{3pt}
	
	\includegraphics[width=0.66\textwidth]{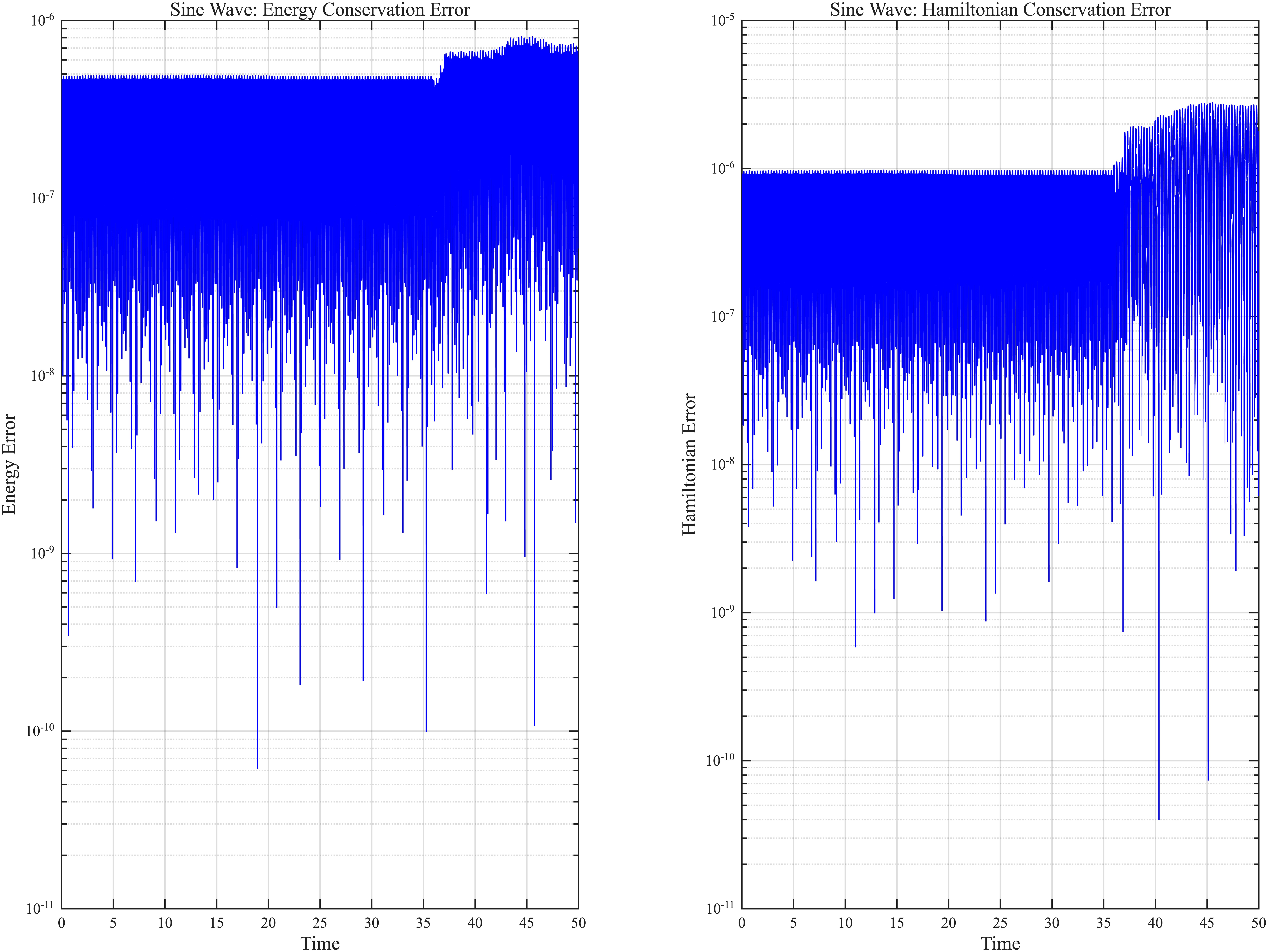}
	
	\vspace{3pt}
	
	\includegraphics[width=0.66\textwidth]{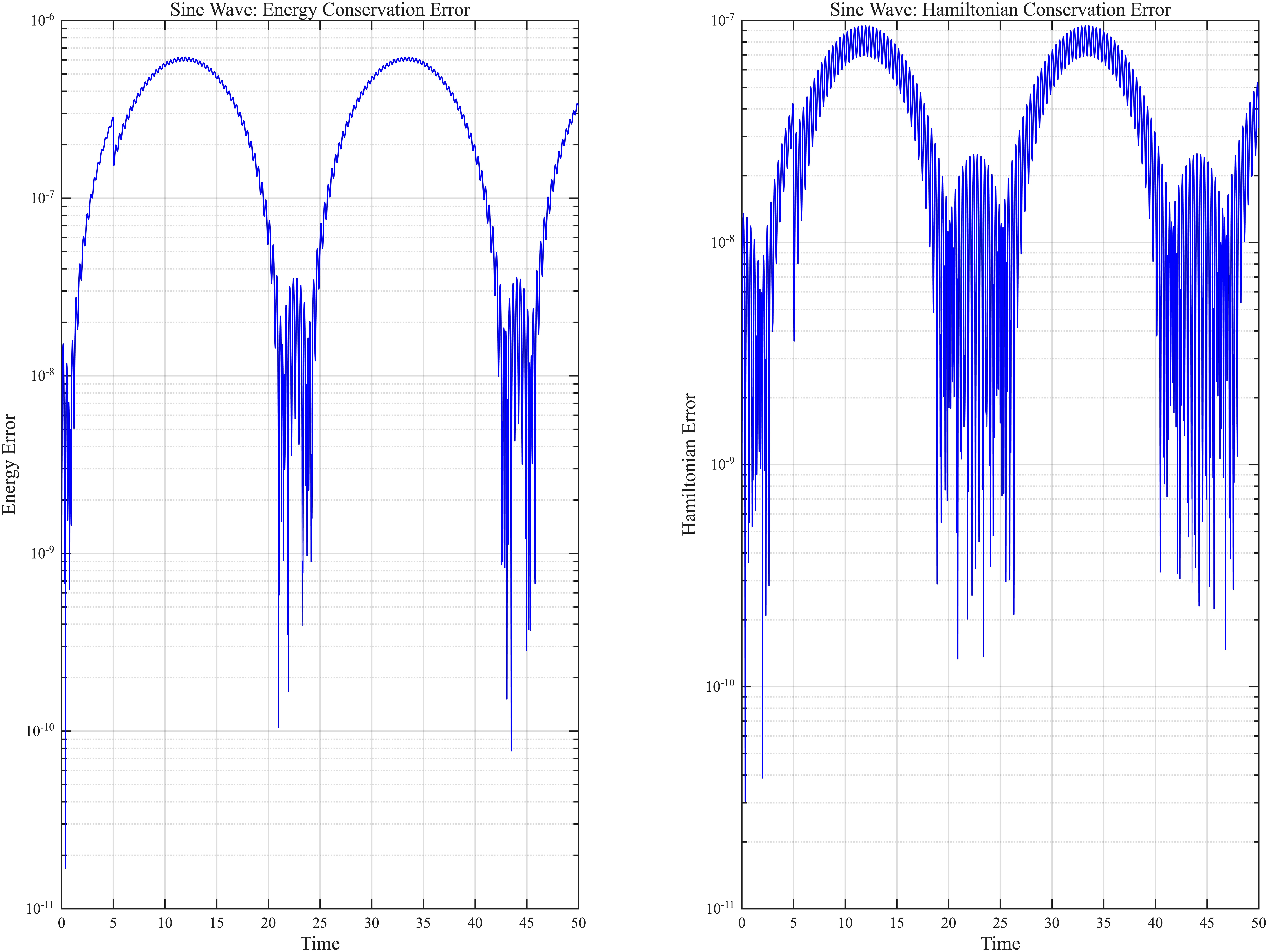}
	
	\caption{Numerical Experiment 2 (third-order nonlinear equation): Errors of energy (left) and Hamiltonian (right) for $\varepsilon = 1$ (top), $0.1$ (middle), and $0.01$ (bottom).}
	\label{fig:5}
\end{figure}

\subsubsection{Experiment 3: Classical KdV Equation with Cnoidal Wave Solution}

In this example, we consider the KdV equation 
$u_t + \varepsilon u_{xxx} + (f(u))_x = 0$ 
with $\varepsilon = \frac{1}{24^2}$ and $f(u) = \frac{u^2}{2}$ on the domain $\Omega = [0,1]$. We test the method using a cnoidal wave solution of the form 
$u(x,t) = A\,\mathrm{cn}^2(z)$,
where $\mathrm{cn}(z) = \mathrm{cn}(z|m)$ denotes the Jacobi elliptic function with modulus $m = 0.9$, and $z = 4K(m)(x - vt - x_0)$; see \cite{bona2013conservative, karakashian2016posteriori, liu2016hamiltonian, zhang2019conservative}.  The amplitude and wave speed are given by $A = 192m\varepsilon K(m)^2$ and $v = 64\varepsilon(2m-1)K(m)^2$, respectively, where $K(m) = \int_{0}^{\pi/2} \frac{d\theta}{\sqrt{1 - m\sin^2\theta}}$ is the complete elliptic integral of the first kind \cite{abramowitz1948handbook}. Since the phase shift parameter $x_0$ is arbitrary, we set $x_0 = 0$. With this choice, the solution $u$ is spatially periodic with period $1$. 

In Table \ref{tab:combined_kdv}, we present the $L^2$-errors of the approximate solutions for $u$, $q$, and $p$ with $k=0,1,2, 3, 4$. The mesh size is \(h=1/N\), where \(N=2^l,\; l=3,\ldots,7\). The final time is \(T=0.1\), and the time step is chosen as \(\Delta t=0.2h\) for \(k=0,2\), \(\Delta t=0.04h\) for \(k=1\), and \(\Delta t=0.01\) for \(k=3,4\). The observed convergence orders are consistent with those obtained in the previous numerical experiments.

Next, we compute the numerical solution using $k=2$ on $32$ intervals over a longer time interval $t\in[0, 50]$. The evolution of the energy and Hamiltonian of the DG solution together with the corresponding errors are shown in Fig. \ref{fig:6}. These results indicate that both the energy and the Hamiltonian are well preserved throughout the entire simulation time.

\begin{table}[H]
	\centering
	\caption{Numerical Experiment 3 (classical KdV equation): \(L^2\)-errors and convergence orders.}
	\label{tab:combined_kdv}
	\renewcommand{\arraystretch}{1.30}
	\begin{tabular}{cc cc cc cc}
		\toprule
		$k$ & $N$
		& \multicolumn{2}{c}{$u_h$}
		& \multicolumn{2}{c}{$q_h$}
		& \multicolumn{2}{c}{$p_h$} \\
		\cmidrule(lr){3-4}
		\cmidrule(lr){5-6}
		\cmidrule(lr){7-8}
		&  & $L_2$ Error & Order & $L_2$ Error & Order & $L_2$ Error & Order \\
		\midrule
		
		\multirow{5}{*}{0}
		& 8   & $5.52\mathrm{e}{-1}$ & --
		& $8.16\mathrm{e}{+0}$ & --
		& $3.39\mathrm{e}{-1}$ & -- \\
		& 16  & $2.70\mathrm{e}{-1}$ & 1.03
		& $4.57\mathrm{e}{+0}$ & 0.84
		& $2.29\mathrm{e}{-1}$ & 0.56 \\
		& 32  & $1.11\mathrm{e}{-1}$ & 1.28
		& $1.92\mathrm{e}{+0}$ & 1.25
		& $8.81\mathrm{e}{-2}$ & 1.38 \\
		& 64  & $4.57\mathrm{e}{-2}$ & 1.28
		& $8.05\mathrm{e}{-1}$ & 1.26
		& $2.98\mathrm{e}{-2}$ & 1.56 \\
		& 128 & $2.21\mathrm{e}{-2}$ & 1.05
		& $3.86\mathrm{e}{-1}$ & 1.06
		& $1.34\mathrm{e}{-2}$ & 1.16 \\
		\midrule
		
		\multirow{5}{*}{1}
		& 8   & $1.57\mathrm{e}{-1}$ & --
		& $2.22\mathrm{e}{+0}$ & --
		& $1.12\mathrm{e}{-1}$ & -- \\
		& 16  & $8.05\mathrm{e}{-2}$ & 0.96
		& $1.67\mathrm{e}{+0}$ & 0.42
		& $7.91\mathrm{e}{-2}$ & 0.50 \\
		& 32  & $2.47\mathrm{e}{-2}$ & 1.70
		& $9.53\mathrm{e}{-1}$ & 0.81
		& $9.97\mathrm{e}{-2}$ & $-0.33$ \\
		& 64  & $6.14\mathrm{e}{-3}$ & 2.01
		& $7.23\mathrm{e}{-1}$ & 0.40
		& $5.20\mathrm{e}{-2}$ & 0.94 \\
		& 128 & $1.93\mathrm{e}{-3}$ & 1.67
		& $3.71\mathrm{e}{-1}$ & 0.96
		& $2.85\mathrm{e}{-2}$ & 0.87 \\
		\midrule
		
		\multirow{4}{*}{2}
		& 8   & $9.84\mathrm{e}{-2}$ & --
		& $2.06\mathrm{e}{+0}$ & --
		& $6.71\mathrm{e}{-2}$ & -- \\
		& 16  & $1.16\mathrm{e}{-2}$ & 3.08
		& $4.89\mathrm{e}{-1}$ & 2.07
		& $6.80\mathrm{e}{-2}$ & $-0.02$ \\
		& 32  & $6.11\mathrm{e}{-4}$ & 4.25
		& $5.53\mathrm{e}{-2}$ & 3.15
		& $1.79\mathrm{e}{-2}$ & 1.92 \\
		& 64  & $5.12\mathrm{e}{-5}$ & 3.58
		& $6.09\mathrm{e}{-3}$ & 3.18
		& $5.14\mathrm{e}{-3}$ & 1.80 \\
		\midrule
		
		\multirow{3}{*}{3}
		& 8   & $1.11\mathrm{e}{-2}$ & --
		& $4.30\mathrm{e}{-1}$ & --
		& $4.67\mathrm{e}{-2}$ & -- \\
		& 16  & $3.88\mathrm{e}{-3}$ & 1.52
		& $1.07\mathrm{e}{-1}$ & 2.01
		& $3.20\mathrm{e}{-2}$ & 0.55 \\
		& 32  & $1.09\mathrm{e}{-4}$ & 5.16
		& $2.20\mathrm{e}{-2}$ & 2.28
		& $9.87\mathrm{e}{-3}$ & 1.70 \\
		\midrule
		
		\multirow{3}{*}{4}
		& 8   & $7.60\mathrm{e}{-3}$ & --
		& $3.16\mathrm{e}{-1}$ & --
		& $4.14\mathrm{e}{-2}$ & -- \\
		& 16  & $1.33\mathrm{e}{-4}$ & 5.84
		& $1.86\mathrm{e}{-2}$ & 4.08
		& $5.93\mathrm{e}{-3}$ & 2.80 \\
		& 32  & $3.30\mathrm{e}{-6}$ & 5.33
		& $6.23\mathrm{e}{-4}$ & 4.90
		& $5.18\mathrm{e}{-4}$ & 3.52 \\
		\bottomrule
	\end{tabular}
\end{table}

\begin{figure}[htbp]
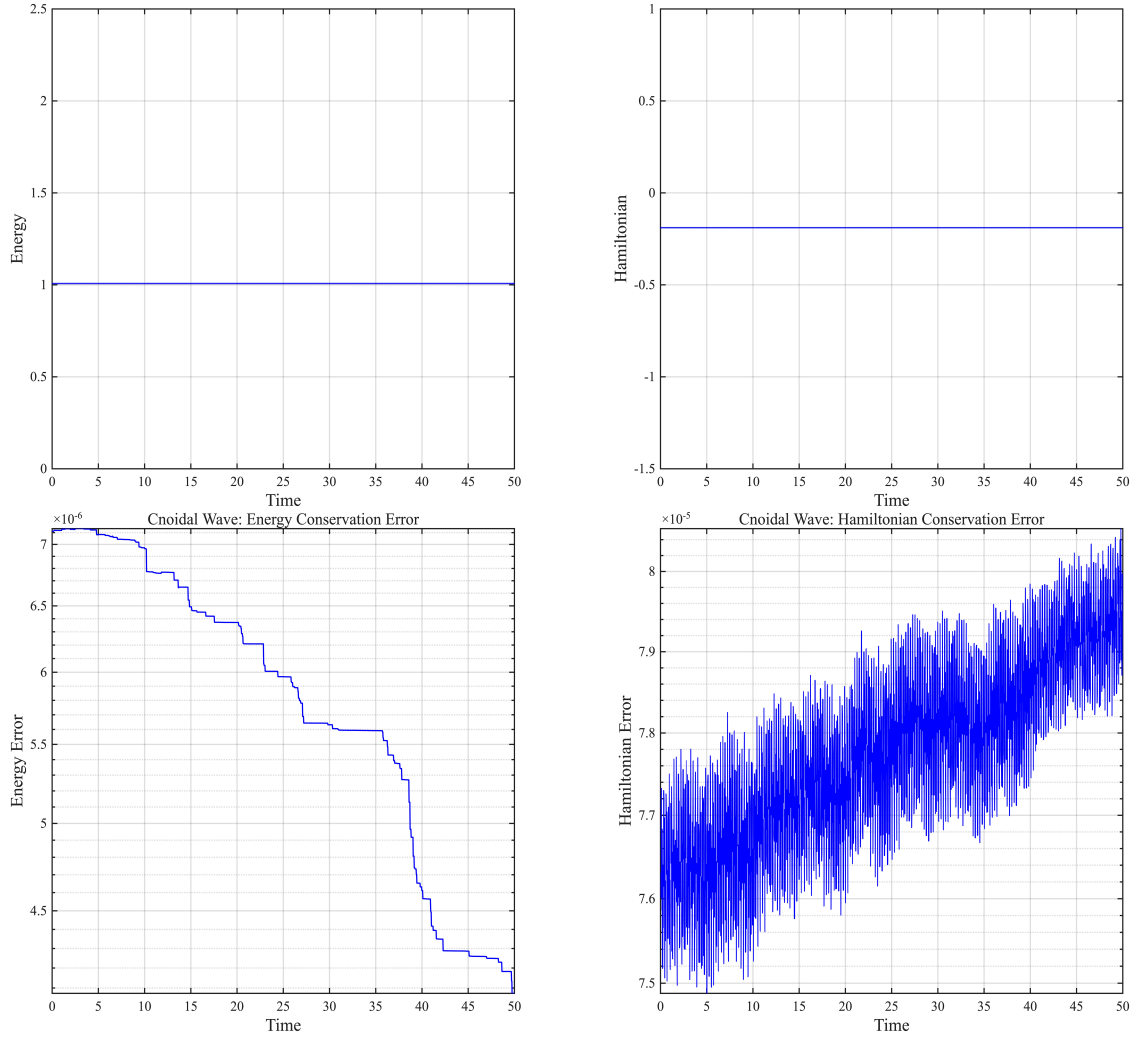

	\centering
	\includegraphics[width=0.9\textwidth]{./Graphics/P3Plot.jpeg}
	\includegraphics[width=0.9\textwidth]{./Graphics/P3Error.jpeg}
	\caption{Numerical Experiment 3 (classical KdV equation): Energy conservation and error (left) and Hamiltonian conservation and error (right) conservation for the Cnoidal wave.}
	\label{fig:6}
\end{figure}

\subsection{Numerical Results for the HS-KdV System}\label{sec:numerics_hs-kdv}
In this subsection, we present numerical experiments to examine the convergence behavior and conservation properties of the proposed DG scheme for the HS--KdV system. Two test problems are considered. In the first example, we consider a forced HS--KdV system with manufactured sinusoidal solutions. In the second example, we consider the HS-KdV system with solitary-wave solutions. For both tests, we investigate the convergence rates and conservation properties of the proposed conservative DG method.

\subsubsection{Experiment 4: Forced HS--KdV System with Sinusoidal Solutions}

In this experiment, we consider the coupled system
\begin{equation*}
    \begin{aligned}
        u_t &= a u_{xxx} + 6a u u_x + 2b v v_x+g_1, \\
        v_t &= -v_{xxx} - 3u v_x+g_2
    \end{aligned}
\end{equation*}
with $a=1, b=1$, and periodic boundary conditions on \(\Omega=[0,1]\). The initial conditions are taken as
$
u(x,0)=\sin(2\pi x)$ and $ v(x,0)=\sin(2\pi x),
$
and $g_1$ and $g_2$ are functions which give the exact solution
\begin{equation*}
            u(x,t) = \sin(2\pi x+t), \qquad         v(x,t) = \sin(2\pi x+t),
\end{equation*}
For the spatial discretization, we use polynomial approximations of degree \(k=0,1,2, 3, 4\). The mesh size is chosen as \(h=1/N\), where \(N=2^l\) with \(l=3,\dots,7\). For the time discretization with IRK4, we take \(\Delta t=0.2h\) for \(k=0\), \(\Delta t=0.04\) for \(k=1\), and \(\Delta t=0.01\) for \(k=2, 3, 4\). The final time is set to \(T=0.1\).

The \(L^2\)-errors and convergence orders in Table \ref{tab:12} indicate that the approximate solutions to the primary variables \(u\) and \(v\) and their first-order derivatives \(q\) and \(w\), attain optimal convergence for even polynomial degrees, whereas suboptimal rates are observed for the odd degree. For auxiliary variables \(p\) and \(r\), the method has optimal convergence rates for \(k=0\) but suboptimal rates for other polynomial degrees. 

We also examine the conservation behavior of the proposed DG scheme. The numerical energy and Hamiltonian obtained with $k=2$ on 32 uniform elements  over the time interval \(t\in[0,50]\) together with their errors are shown in Fig. \ref{fig:hskdv1_plot}. The results indicate that the proposed method preserves these invariants to high accuracy over long-time integration. Since the exact solution in this example is manufactured, this test mainly serves to assess both the accuracy of the scheme and its ability to maintain the expected conservative structure at the discrete level.

\begin{table}[H]
\centering
\caption{Numerical Experiment 4 (forced HS--KdV system with sinusoidal solutions): \(L^2\)-errors and convergence orders.}
\label{tab:12}
\renewcommand{\arraystretch}{1.20}
\resizebox{\textwidth}{!}{
\begin{tabular}{cc cc cc cc|cc cc cc}
\toprule
$k$ & $N$
& \multicolumn{2}{c}{$u_h$}
& \multicolumn{2}{c}{$q_h$}
& \multicolumn{2}{c|}{$p_h$}
& \multicolumn{2}{c}{$v_h$}
& \multicolumn{2}{c}{$w_h$}
& \multicolumn{2}{c}{$r_h$} \\
\cmidrule(lr){3-4}
\cmidrule(lr){5-6}
\cmidrule(lr){7-8}
\cmidrule(lr){9-10}
\cmidrule(lr){11-12}
\cmidrule(lr){13-14}
 &  & $L_2$ Error & Order
    & $L_2$ Error & Order
    & $L_2$ Error & Order
    & $L_2$ Error & Order
    & $L_2$ Error & Order
    & $L_2$ Error & Order \\
\midrule

\multirow{5}{*}{0}
& 8   & $6.05\mathrm{e}{-1}$ & -- & $3.05\mathrm{e}{0}$ & -- & $1.54\mathrm{e}{1}$ & -- & $5.02\mathrm{e}{-1}$ & -- & $2.47\mathrm{e}{0}$ & -- & $1.23\mathrm{e}{1}$ & -- \\
& 16  & $1.62\mathrm{e}{-1}$ & 1.90 & $9.09\mathrm{e}{-1}$ & 1.75 & $5.09\mathrm{e}{0}$ & 1.60 & $1.38\mathrm{e}{-1}$ & 1.86 & $7.62\mathrm{e}{-1}$ & 1.70 & $4.24\mathrm{e}{0}$ & 1.54 \\
& 32  & $4.25\mathrm{e}{-2}$ & 1.93 & $2.68\mathrm{e}{-1}$ & 1.76 & $1.71\mathrm{e}{0}$ & 1.57 & $4.24\mathrm{e}{-2}$ & 1.70 & $2.63\mathrm{e}{-1}$ & 1.53 & $1.65\mathrm{e}{0}$ & 1.36 \\
& 64  & $2.01\mathrm{e}{-2}$ & 1.08 & $1.26\mathrm{e}{-1}$ & 1.08 & $8.00\mathrm{e}{-1}$ & 1.10 & $2.01\mathrm{e}{-2}$ & 1.08 & $1.26\mathrm{e}{-1}$ & 1.06 & $7.97\mathrm{e}{-1}$ & 1.05 \\
& 128 & $1.00\mathrm{e}{-2}$ & 1.00 & $6.30\mathrm{e}{-2}$ & 1.00 & $3.97\mathrm{e}{-1}$ & 1.01 & $1.00\mathrm{e}{-2}$ & 1.01 & $6.30\mathrm{e}{-2}$ & 1.01 & $3.96\mathrm{e}{-1}$ & 1.01 \\
\midrule

\multirow{5}{*}{1}
& 8   & $5.13\mathrm{e}{-2}$ & -- & $9.30\mathrm{e}{-1}$ & -- & $1.04\mathrm{e}{1}$ & -- & $4.82\mathrm{e}{-2}$ & -- & $9.33\mathrm{e}{-1}$ & -- & $1.03\mathrm{e}{1}$ & -- \\
& 16  & $1.40\mathrm{e}{-2}$ & 1.87 & $4.91\mathrm{e}{-1}$ & 0.92 & $6.06\mathrm{e}{0}$ & 0.78 & $1.25\mathrm{e}{-2}$ & 1.95 & $4.93\mathrm{e}{-1}$ & 0.92 & $6.02\mathrm{e}{0}$ & 0.78 \\
& 32  & $4.55\mathrm{e}{-3}$ & 1.62 & $2.49\mathrm{e}{-1}$ & 0.98 & $3.15\mathrm{e}{0}$ & 0.94 & $3.60\mathrm{e}{-3}$ & 1.80 & $2.50\mathrm{e}{-1}$ & 0.98 & $3.13\mathrm{e}{0}$ & 0.94 \\
& 64  & $1.85\mathrm{e}{-3}$ & 1.30 & $1.25\mathrm{e}{-1}$ & 1.00 & $1.59\mathrm{e}{0}$ & 0.99 & $1.27\mathrm{e}{-3}$ & 1.50 & $1.26\mathrm{e}{-1}$ & 0.99 & $1.58\mathrm{e}{0}$ & 0.99 \\
& 128 & $9.59\mathrm{e}{-4}$ & 0.95 & $6.27\mathrm{e}{-2}$ & 1.00 & $7.98\mathrm{e}{-1}$ & 0.99 & $5.50\mathrm{e}{-4}$ & 1.21 & $6.29\mathrm{e}{-2}$ & 1.00 & $7.91\mathrm{e}{-1}$ & 1.00 \\
\midrule

\multirow{4}{*}{2}
& 8   & $1.08\mathrm{e}{-3}$ & -- & $2.85\mathrm{e}{-2}$ & -- & $1.83\mathrm{e}{0}$ & -- & $1.07\mathrm{e}{-3}$ & -- & $2.85\mathrm{e}{-2}$ & -- & $1.83\mathrm{e}{0}$ & -- \\
& 16  & $1.35\mathrm{e}{-4}$ & 3.00 & $3.47\mathrm{e}{-3}$ & 3.04 & $4.27\mathrm{e}{-1}$ & 2.10 & $1.35\mathrm{e}{-4}$ & 3.00 & $3.47\mathrm{e}{-3}$ & 3.04 & $4.27\mathrm{e}{-1}$ & 2.10 \\
& 32  & $1.72\mathrm{e}{-5}$ & 2.97 & $4.31\mathrm{e}{-4}$ & 3.01 & $1.04\mathrm{e}{-1}$ & 2.04 & $1.69\mathrm{e}{-5}$ & 3.00 & $4.31\mathrm{e}{-4}$ & 3.01 & $1.04\mathrm{e}{-1}$ & 2.04 \\
& 64  & $2.11\mathrm{e}{-6}$ & 3.03 & $5.37\mathrm{e}{-5}$ & 3.00 & $2.58\mathrm{e}{-2}$ & 2.01 & $2.11\mathrm{e}{-6}$ & 3.00 & $5.37\mathrm{e}{-5}$ & 3.00 & $2.58\mathrm{e}{-2}$ & 2.01 \\
\midrule

\multirow{3}{*}{3}
& 8   & $6.55\mathrm{e}{-5}$ & -- & $7.32\mathrm{e}{-3}$ & -- & $3.96\mathrm{e}{-1}$ & -- & $5.63\mathrm{e}{-5}$ & -- & $7.34\mathrm{e}{-3}$ & -- & $3.96\mathrm{e}{-1}$ & -- \\
& 16  & $5.38\mathrm{e}{-6}$ & 3.61 & $9.86\mathrm{e}{-4}$ & 2.89 & $7.98\mathrm{e}{-2}$ & 2.31 & $3.55\mathrm{e}{-6}$ & 3.99 & $9.88\mathrm{e}{-4}$ & 2.89 & $7.98\mathrm{e}{-2}$ & 2.31 \\
& 32  & $5.58\mathrm{e}{-7}$ & 3.27 & $1.25\mathrm{e}{-4}$ & 2.98 & $1.80\mathrm{e}{-2}$ & 2.15 & $2.40\mathrm{e}{-7}$ & 3.88 & $1.26\mathrm{e}{-4}$ & 2.97 & $1.80\mathrm{e}{-2}$ & 2.15 \\
\midrule

\multirow{3}{*}{4}
& 8   & $2.10\mathrm{e}{-6}$ & -- & $1.28\mathrm{e}{-4}$ & -- & $1.96\mathrm{e}{-2}$ & -- & $2.10\mathrm{e}{-6}$ & -- & $1.28\mathrm{e}{-4}$ & -- & $1.95\mathrm{e}{-2}$ & -- \\
& 16  & $6.61\mathrm{e}{-8}$ & 4.99 & $3.98\mathrm{e}{-6}$ & 5.01 & $1.23\mathrm{e}{-3}$ & 3.99 & $6.62\mathrm{e}{-8}$ & 4.99 & $3.98\mathrm{e}{-6}$ & 5.00 & $1.23\mathrm{e}{-3}$ & 3.99 \\
\bottomrule
\end{tabular}}
\end{table}

\begin{table}[H]
\centering
\caption{Numerical Experiment 5 (HS--KdV System with Solitary-Wave Solutions): $L^2$-errors and convergence orders.}
\label{tab:new}
\renewcommand{\arraystretch}{1.20}
\resizebox{\textwidth}{!}{
\begin{tabular}{cc cc cc cc|cc cc cc}
\toprule
$k$ & $N$
& \multicolumn{2}{c}{$u_h$}
& \multicolumn{2}{c}{$q_h$}
& \multicolumn{2}{c|}{$p_h$}
& \multicolumn{2}{c}{$v_h$}
& \multicolumn{2}{c}{$w_h$}
& \multicolumn{2}{c}{$r_h$} \\
\cmidrule(lr){3-4}
\cmidrule(lr){5-6}
\cmidrule(lr){7-8}
\cmidrule(lr){9-10}
\cmidrule(lr){11-12}
\cmidrule(lr){13-14}
 &  & $L_2$ Error & Order
    & $L_2$ Error & Order
    & $L_2$ Error & Order
    & $L_2$ Error & Order
    & $L_2$ Error & Order
    & $L_2$ Error & Order \\
\midrule

\multirow{6}{*}{0}
& 8   & $7.09\mathrm{e}{-1}$ & -- & $3.58\mathrm{e}{-1}$ & -- & $7.71\mathrm{e}{-1}$ & -- & $2.24\mathrm{e}{-1}$ & -- & $7.88\mathrm{e}{-2}$ & -- & $4.93\mathrm{e}{-2}$ & -- \\
& 16  & $5.82\mathrm{e}{-1}$ & 0.28 & $3.36\mathrm{e}{-1}$ & 0.09 & $6.50\mathrm{e}{-1}$ & 0.25 & $1.54\mathrm{e}{-1}$ & 0.54 & $6.87\mathrm{e}{-2}$ & 0.20 & $4.90\mathrm{e}{-2}$ & 0.01 \\
& 32  & $3.40\mathrm{e}{-1}$ & 0.78 & $2.74\mathrm{e}{-1}$ & 0.30 & $3.73\mathrm{e}{-1}$ & 0.80 & $7.56\mathrm{e}{-2}$ & 1.03 & $4.85\mathrm{e}{-2}$ & 0.50 & $4.41\mathrm{e}{-2}$ & 0.15 \\
& 64  & $1.60\mathrm{e}{-1}$ & 1.08 & $1.63\mathrm{e}{-1}$ & 0.74 & $2.40\mathrm{e}{-1}$ & 0.63 & $3.71\mathrm{e}{-2}$ & 1.03 & $2.55\mathrm{e}{-2}$ & 0.93 & $2.82\mathrm{e}{-2}$ & 0.65 \\
& 128 & $8.18\mathrm{e}{-2}$ & 0.97 & $7.61\mathrm{e}{-2}$ & 1.10 & $1.18\mathrm{e}{-1}$ & 1.02 & $1.88\mathrm{e}{-2}$ & 0.98 & $1.19\mathrm{e}{-2}$ & 1.10 & $1.31\mathrm{e}{-2}$ & 1.11 \\
& 256 & $4.11\mathrm{e}{-2}$ & 0.99 & $3.58\mathrm{e}{-2}$ & 1.09 & $4.90\mathrm{e}{-2}$ & 1.27 & $9.40\mathrm{e}{-3}$ & 1.00 & $5.69\mathrm{e}{-3}$ & 1.06 & $5.60\mathrm{e}{-3}$ & 1.22 \\
\midrule

\multirow{6}{*}{1}
& 8   & $4.56\mathrm{e}{-1}$ & -- & $2.95\mathrm{e}{-1}$ & -- & $5.63\mathrm{e}{-1}$ & -- & $1.05\mathrm{e}{-1}$ & -- & $5.47\mathrm{e}{-2}$ & -- & $4.66\mathrm{e}{-2}$ & -- \\
& 16  & $2.15\mathrm{e}{-1}$ & 1.09 & $2.21\mathrm{e}{-1}$ & 0.42 & $2.84\mathrm{e}{-1}$ & 0.99 & $3.10\mathrm{e}{-2}$ & 1.76 & $3.81\mathrm{e}{-2}$ & 0.52 & $3.46\mathrm{e}{-2}$ & 0.43 \\
& 32  & $6.43\mathrm{e}{-2}$ & 1.74 & $1.31\mathrm{e}{-1}$ & 0.75 & $1.13\mathrm{e}{-1}$ & 1.33 & $1.28\mathrm{e}{-2}$ & 1.28 & $2.23\mathrm{e}{-2}$ & 0.77 & $1.44\mathrm{e}{-2}$ & 1.27 \\
& 64  & $3.49\mathrm{e}{-2}$ & 0.88 & $9.16\mathrm{e}{-2}$ & 0.52 & $1.23\mathrm{e}{-1}$ & -0.13 & $5.52\mathrm{e}{-3}$ & 1.21 & $1.63\mathrm{e}{-2}$ & 0.45 & $2.17\mathrm{e}{-2}$ & -0.59 \\
& 128 & $7.90\mathrm{e}{-3}$ & 2.14 & $6.07\mathrm{e}{-2}$ & 0.59 & $1.13\mathrm{e}{-1}$ & 0.13 & $1.39\mathrm{e}{-3}$ & 1.99 & $1.03\mathrm{e}{-2}$ & 0.66 & $1.60\mathrm{e}{-2}$ & 0.44 \\
& 256 & $1.81\mathrm{e}{-3}$ & 2.13 & $3.36\mathrm{e}{-2}$ & 0.85 & $7.38\mathrm{e}{-2}$ & 0.61 & $3.80\mathrm{e}{-4}$ & 1.87 & $5.47\mathrm{e}{-3}$ & 0.92 & $9.27\mathrm{e}{-3}$ & 0.79 \\
\midrule

\multirow{6}{*}{2}
& 8   & $2.12\mathrm{e}{-1}$ & -- & $2.01\mathrm{e}{-1}$ & -- & $2.73\mathrm{e}{-1}$ & -- & $3.02\mathrm{e}{-2}$ & -- & $3.64\mathrm{e}{-2}$ & -- & $3.47\mathrm{e}{-2}$ & -- \\
& 16  & $7.30\mathrm{e}{-2}$ & 1.54 & $1.69\mathrm{e}{-1}$ & 0.25 & $1.30\mathrm{e}{-1}$ & 1.07 & $1.54\mathrm{e}{-2}$ & 0.97 & $3.34\mathrm{e}{-2}$ & 0.12 & $2.01\mathrm{e}{-2}$ & 0.79 \\
& 32  & $3.29\mathrm{e}{-2}$ & 1.15 & $1.18\mathrm{e}{-1}$ & 0.52 & $1.10\mathrm{e}{-1}$ & 0.24 & $5.39\mathrm{e}{-3}$ & 1.52 & $1.55\mathrm{e}{-2}$ & 1.11 & $1.99\mathrm{e}{-2}$ & 0.02 \\
& 64  & $5.55\mathrm{e}{-3}$ & 2.57 & $2.49\mathrm{e}{-2}$ & 2.25 & $9.50\mathrm{e}{-2}$ & 0.21 & $1.83\mathrm{e}{-3}$ & 1.56 & $3.88\mathrm{e}{-3}$ & 2.00 & $1.24\mathrm{e}{-2}$ & 0.68 \\
& 128 & $8.54\mathrm{e}{-4}$ & 2.70 & $4.23\mathrm{e}{-3}$ & 2.56 & $3.62\mathrm{e}{-2}$ & 1.39 & $8.74\mathrm{e}{-5}$ & 4.38 & $3.84\mathrm{e}{-4}$ & 3.34 & $3.53\mathrm{e}{-3}$ & 1.81 \\
& 256 & $7.51\mathrm{e}{-5}$ & 3.51 & $4.47\mathrm{e}{-4}$ & 3.24 & $8.70\mathrm{e}{-3}$ & 2.06 & $8.70\mathrm{e}{-6}$ & 3.33 & $4.15\mathrm{e}{-5}$ & 3.21 & $8.23\mathrm{e}{-4}$ & 2.10 \\
\midrule

\multirow{6}{*}{3}
& 8   & $9.87\mathrm{e}{-2}$ & -- & $1.88\mathrm{e}{-1}$ & -- & $1.52\mathrm{e}{-1}$ & -- & $2.08\mathrm{e}{-2}$ & -- & $3.85\mathrm{e}{-2}$ & -- & $2.69\mathrm{e}{-2}$ & -- \\
& 16  & $5.42\mathrm{e}{-2}$ & 0.87 & $1.48\mathrm{e}{-1}$ & 0.34 & $1.39\mathrm{e}{-1}$ & 0.14 & $8.34\mathrm{e}{-3}$ & 1.32 & $1.88\mathrm{e}{-2}$ & 1.04 & $2.37\mathrm{e}{-2}$ & 0.18 \\
& 32  & $6.70\mathrm{e}{-3}$ & 3.02 & $2.50\mathrm{e}{-2}$ & 2.57 & $7.72\mathrm{e}{-2}$ & 0.85 & $1.56\mathrm{e}{-3}$ & 2.42 & $3.15\mathrm{e}{-3}$ & 2.57 & $8.74\mathrm{e}{-3}$ & 1.44 \\
& 64  & $8.01\mathrm{e}{-4}$ & 3.06 & $3.17\mathrm{e}{-3}$ & 2.98 & $2.10\mathrm{e}{-2}$ & 1.88 & $9.99\mathrm{e}{-5}$ & 3.97 & $3.72\mathrm{e}{-4}$ & 3.08 & $2.50\mathrm{e}{-3}$ & 1.80 \\
& 128 & $1.01\mathrm{e}{-4}$ & 2.98 & $6.66\mathrm{e}{-4}$ & 2.25 & $8.07\mathrm{e}{-3}$ & 1.38 & $1.16\mathrm{e}{-5}$ & 3.10 & $7.06\mathrm{e}{-5}$ & 2.40 & $7.81\mathrm{e}{-4}$ & 1.68 \\
& 256 & $1.19\mathrm{e}{-5}$ & 3.09 & $9.51\mathrm{e}{-5}$ & 2.81 & $1.68\mathrm{e}{-3}$ & 2.27 & $1.45\mathrm{e}{-6}$ & 3.01 & $1.01\mathrm{e}{-5}$ & 2.81 & $1.65\mathrm{e}{-4}$ & 2.24 \\
\midrule

\multirow{5}{*}{4}
& 8   & $7.81\mathrm{e}{-2}$ & -- & $1.93\mathrm{e}{-1}$ & -- & $1.44\mathrm{e}{-1}$ & -- & $1.47\mathrm{e}{-2}$ & -- & $2.98\mathrm{e}{-2}$ & -- & $2.87\mathrm{e}{-2}$ & -- \\
& 16  & $2.07\mathrm{e}{-2}$ & 1.91 & $6.26\mathrm{e}{-2}$ & 1.62 & $1.17\mathrm{e}{-1}$ & 0.30 & $3.19\mathrm{e}{-3}$ & 2.20 & $5.57\mathrm{e}{-3}$ & 2.42 & $1.53\mathrm{e}{-2}$ & 0.91 \\
& 32  & $2.86\mathrm{e}{-3}$ & 2.86 & $1.16\mathrm{e}{-2}$ & 2.43 & $4.11\mathrm{e}{-2}$ & 1.51 & $3.76\mathrm{e}{-4}$ & 3.09 & $1.51\mathrm{e}{-3}$ & 1.89 & $6.38\mathrm{e}{-3}$ & 1.26 \\
& 64  & $4.48\mathrm{e}{-4}$ & 2.67 & $1.37\mathrm{e}{-3}$ & 3.08 & $8.48\mathrm{e}{-3}$ & 2.28 & $1.03\mathrm{e}{-4}$ & 1.87 & $3.65\mathrm{e}{-4}$ & 2.04 & $2.56\mathrm{e}{-3}$ & 1.32 \\
& 128 & $8.13\mathrm{e}{-6}$ & 5.78 & $6.61\mathrm{e}{-5}$ & 4.38 & $1.18\mathrm{e}{-3}$ & 2.85 & $6.20\mathrm{e}{-7}$ & 7.37 & $4.50\mathrm{e}{-6}$ & 6.34 & $8.66\mathrm{e}{-5}$ & 4.89 \\
\bottomrule
\end{tabular}}
\end{table}

\begin{figure}[h]
    \centering
    \includegraphics[width=0.9\textwidth]{./Graphics/P1HSKdVPlot.jpeg}
    \includegraphics[width=0.9\textwidth]{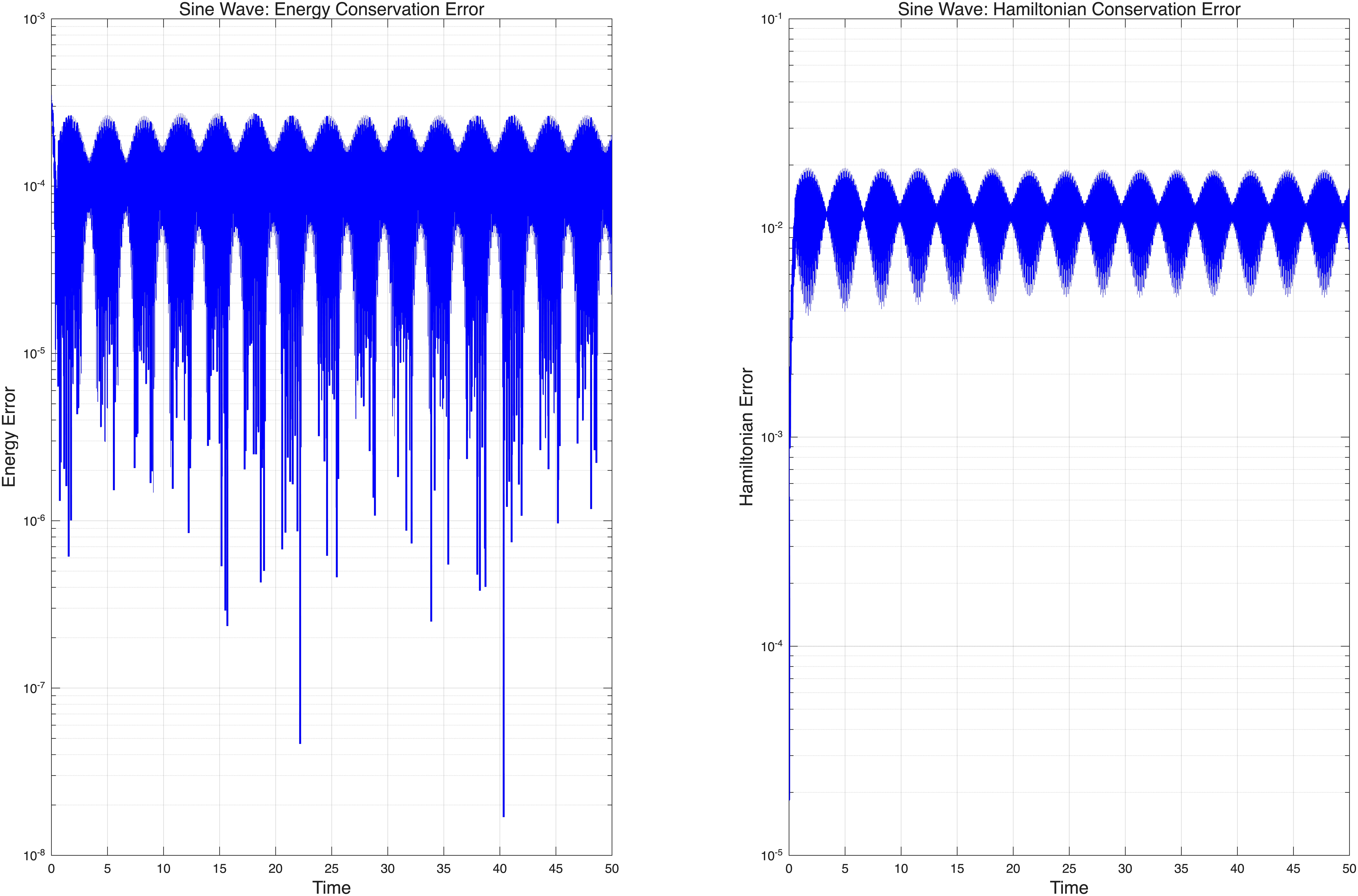}
    \caption{Numerical Experiment 4 (forced HS--KdV system with sinusoidal solutions): Conservation of energy with corresponding error (left) and Hamiltonian conservation with corresponding error (right).}
    \label{fig:hskdv1_plot}
\end{figure}

\subsubsection{Experiment 5: HS--KdV System with Solitary-Wave Solutions}

In this experiment, we consider the HS--KdV system \eqref{eq:hskdv} with parameters \(a=-0.125\) and \(b=-3\); see \cite{ismail2014numerical}. The computational domain is taken as \(\Omega=[-50,50]\), which is sufficiently large to accommodate the solitary wave under periodic boundary conditions. The initial conditions are given by
\begin{equation*}
u_0 = 2\lambda^2 \sech^2(\xi), \quad v_0 = \frac{1}{2\sqrt{\omega}} \sech(\xi),\quad \textrm{ where } \quad
\xi = \lambda x + \frac{1}{2\log(\omega)}, \quad
\omega = -\frac{b}{8(4a+1)\lambda^4}, \quad \lambda = 0.5.
\end{equation*}
The corresponding exact solution is
\begin{equation*}
u(x,t) = 2\lambda^2 \sech^2(\xi), \quad
v(x,t) = \frac{1}{2\sqrt{\omega}} \sech(\xi), \quad \textrm{ where } \quad
\xi = \lambda(x-\lambda^2 t) + \frac{1}{2\log(\omega)}, \quad
\omega = -\frac{b}{8(4a+1)\lambda^4}.
\end{equation*}

We next examine the convergence of the proposed DG method for this solitary-wave problem. Polynomial degrees \(k=0,1,2, 3, 4\) are employed in the spatial discretization. The mesh size is chosen as \(h=100/N\), where \(N=2^l\), \(l=3,\ldots,7\), and the time step is taken as \(\Delta t=0.01\). 

The corresponding \(L^2\)-errors and convergence orders are reported in Table \ref{tab:new}. We observe that the approximate solutions to the primary variables $u$ and $v$, as well as their first-order spatial derivatives $q$ and $w$, have optimal convergence rates for even polynomial degrees $k=0, 2$ and suboptimal convergence rates for odd degrees $k=1, 3$. This is typical for energy-conserving DG methods. For the variables $p$ and $r$ which involve second-order spatial derivatives, the convergence behavior is generally suboptimal and somewhat irregular. 

We also examine the conservation behavior of the proposed DG scheme. The numerical energy and Hamiltonian computed with $k=2$ on 32 uniform elements over the time interval \(t\in[0,50]\) together with their errors are shown in Fig. \ref{fig:hskdv2_plot}. The results indicate that the proposed method preserves these invariants to reasonable accuracy over long-time integration. 

\begin{figure}[h]
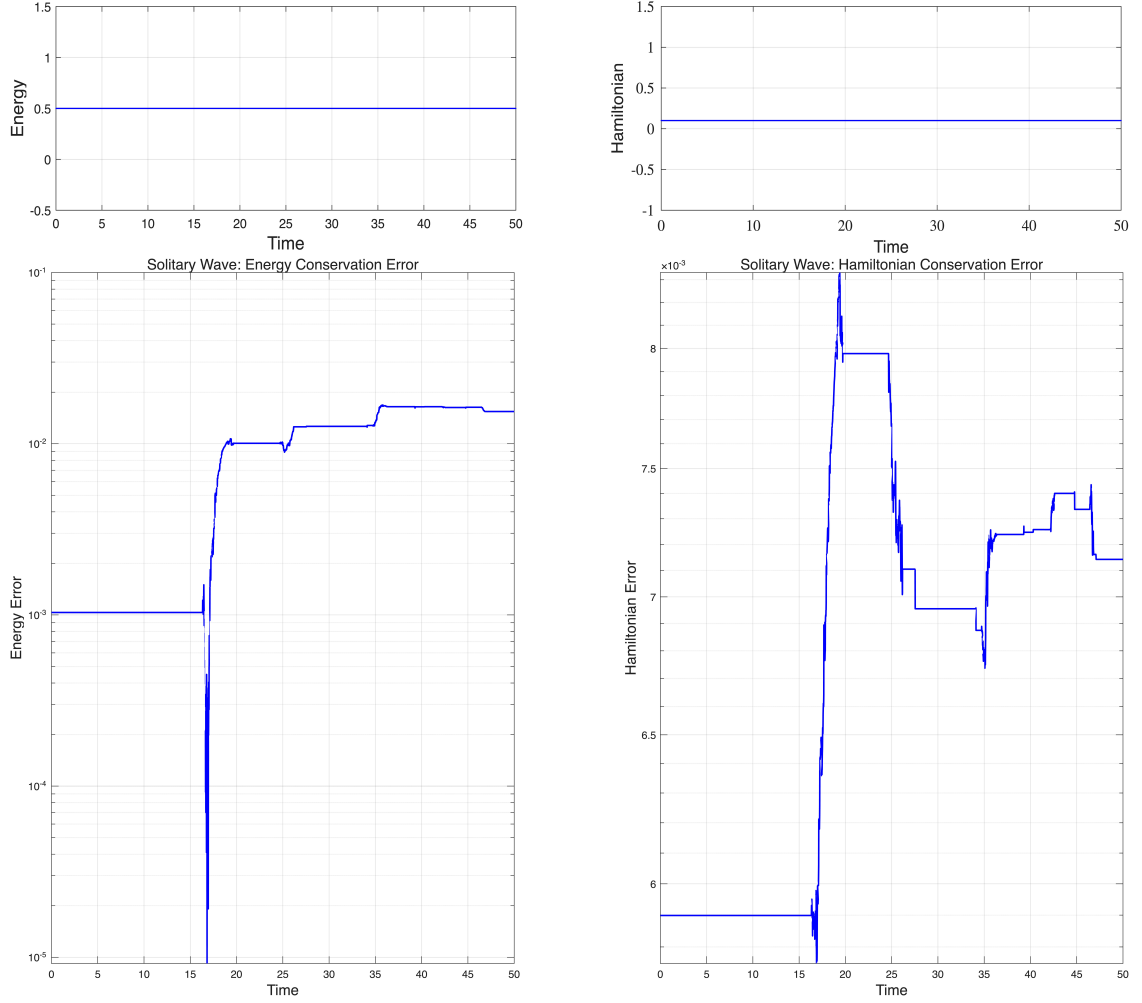

    \centering
    \includegraphics[width=0.9\textwidth]{./Graphics/P2K2N32Plot.jpeg}
    \includegraphics[width=0.9\textwidth]{./Graphics/P2K2N32EHError.jpeg}
    \caption{Numerical Experiment 5 (HS--KdV system with solitary-wave solutions): Conservation of energy with corresponding error (left) and Hamiltonian conservation with corresponding error (right).}
    \label{fig:hskdv2_plot}
\end{figure}

\section{Conclusion}\label{sec:conclusion}
In this paper, we present a new conservative DG method for the gKdV equation that simultaneously conserves the mass, $L^2$ energy, and Hamiltonian of the solution. By carefully designing the numerical traces with implicitly defined penalization parameters, our DG formulation eliminates the time-derivative of a jump term in \cite{ChenDongPereira2022}, allowing for seamless integration with high-order time discretizations while significantly improving computational efficiency. 
Furthermore, we extend the implicit penalization framework to the HS--KdV system and obtain, to the best of our knowledge, the first DG discretization capable of simultaneously preserving three invariants of the system. Numerical results verify that both methods achieve optimal convergence rates for even polynomial degrees and suboptimal rates for odd degrees in the primary variables, and demonstrate the robustness of the proposed schemes for long-time simulations.

Building on the success of the current framework, our future work will focus on extending the conservative DG method via implicit penalization to a broader range of dispersive problems. We plan to apply this approach to the Benjamin--Bona--Mahony (BBM) equation and the Schr\"odinger equation, as well as the Schr\"odinger--KdV system, which will serve as a stepping stone to generalize these techniques to more complex coupled systems and higher-dimensional settings.

\bibliographystyle{plain}
\bibliography{mybibtext}


\appendix
\section{Proof of Lemma \ref{lemma:nonlinear-coupling}}\label{appendix:proof_lemma}  

We write $S$ as
\begin{equation}\label{eq:Snew}
S=S_1+S_2- 6b\left(u_h w_h,\Pi (u_h v_h)\right)- 6b\left(u_h w_h,r_h\right),
\end{equation}
where 
\begin{equation*}
S_1=2b(v_{hx}, v_h p_h),\quad
S_2= 2b\left(r_h,\Pi (u_h v_h)_x\right) - 2b\langle \widehat{r}_h,\Pi (u_h v_h) n\rangle.
\end{equation*}
Applying integration by parts to  \eqref{4.3d} and taking  $\xi=2b\Pi(v_h p_h)$, we obtain
\begin{align*}
    S_1&=2b\left(w_h, \Pi(v_h p_h)\right)-2b\langle \hat{v}_h-v_h, \Pi(v_h p_h)n\rangle\\
    &=2b(w_h, v_h p_h)-2b\langle \hat{v}_h-v_h, \Pi(v_h p_h)n\rangle,
 \end{align*} 
where we used the orthogonality property of the $L^2$ projection operator $\Pi$. This shows that $v_{hx}$ is replaced by $w_h$ in the volume integral, up to a boundary correction term.

Similarly, using \eqref{4.3b} with $\beta=\Pi(w_h v_h)$ and integrating by parts, we obtain
\begin{align*}
    S_1&=2b(q_{hx}, w_h v_h)+2b\langle\hat{q}_h-q_h, \Pi(w_h v_h) n\rangle +6b\left(u_h^2, \Pi(w_h v_h)\right)-2b\langle \hat{v}_h-v_h, \Pi(v_h p_h)n\rangle.
\end{align*} 
This corresponds to replacing $p_h$ by $q_{hx}+3u_h^2$ in the volume integral  with additional boundary terms. 

Next, we rewrite the first term of $S_1$. Applying integration by parts, choosing test functions $\psi=\Pi(q_h v_h)$ in
\eqref{4.3e} and $\xi=\Pi(q_h w_h)$ in \eqref{4.3d}, 
we obtain
\begin{equation}\label{eq:1st_term}
\begin{aligned}
    2b(q_{hx}, w_h v_h)=&-2b(q_h, w_{hx} v_h +w_h v_{hx})+2b\langle q_h, w_h v_h n\rangle\\
    =&-2b(w_{hx}, q_h v_h)-2b(v_{hx}, q_h w_h)+2b\langle q_h, w_h v_h n\rangle\\
    =&-2b(r_h, q_h v_h)+2b\langle\hat{w}_h-w_h, \Pi(q_h v_h) n\rangle \\
    &-2b(w_h, q_h w_h)+2b\langle \hat{v}_h-v_h, \Pi(q_h w_h)n\rangle\\
    &+2b\langle q_h, w_h v_h n\rangle.
\end{aligned}
\end{equation}
Let us rewrite the third term on the right hand side. Taking $\alpha=\Pi(w_h^2)$ in \eqref{4.3a}, applying integration by parts, and taking $\psi=\Pi(u_h w_h)$ in \eqref{4.3e} to get
\begin{equation}
\begin{aligned}\label{eq:3rd_term}
    -2b(w_h, q_h w_h)=&-2b(u_{hx},w_h^2)-2b\langle \hat{u}_h-u_h, \Pi(w_h^2) n\rangle\\
    =& 4b(u_h, w_h w_{hx})-2b\langle u_h, w_h^2 n\rangle-2b\langle \hat{u}_h-u_h, \Pi(w_h^2) n\rangle\\
   =& 4b(r_h, u_h w_h)-4b\langle \hat{w}_h-w_h, \Pi(u_h w_h)n\rangle-2b\langle u_h, w_h^2 n\rangle-2b\langle \hat{u}_h-u_h, \Pi(w_h^2) n\rangle.
\end{aligned}
\end{equation}
Substituting \eqref{eq:3rd_term} into \eqref{eq:1st_term}, and then inserting the result into $S_1$, we obtain
\begin{equation}\label{eq:S1}
\begin{aligned}
S_1=&-2b(r_h, q_h v_h)+4b(r_h, u_h w_h) +6b\left(u_h^2, \Pi(w_h v_h)\right)\\
&-2b\langle \hat{v}_h-v_h, \Pi(v_h p_h)n\rangle+ 2b
\mathcal{T}(q_h, w_h, v_h) -2b\mathcal{T}(u_h, w_h, w_h),   
\end{aligned}
\end{equation}
where $$\mathcal{T}(\phi_1,\phi_2,\phi_3):=\langle \hat{\phi}_1-\phi_1, \Pi(\phi_2 \phi_3) n\rangle +\langle \hat{\phi}_2-\phi_2, \Pi(\phi_1 \phi_3) n\rangle +\langle \hat{\phi}_3-\phi_3, \Pi(\phi_1 \phi_2) n\rangle.
$$

Now we rewrite $S_2$. Using integration by parts, we have
\begin{align*}
S_2=&-2b(r_{hx}, u_h v_h)-2b\langle \hat{r}_h -r_h, \Pi(u_h v_h) n\rangle\\
   =&2b(r_h, u_{hx} v_h +u_h v_{hx})-2b\langle r_h, u_h v_h n\rangle -2b\langle \hat{r}_h -r_h, \Pi(u_h v_h) n\rangle.
   \end{align*}
Taking $\alpha=\Pi(r_h v_h)$ in \eqref{4.3a} and $\xi=\Pi (r_h u_u)$ in \eqref{4.3d}, we obtain
\begin{equation}\label{eq:S2}
        S_2=2b(q_h, r_h v_h)+2b(w_h, r_h u_h)-2b\mathcal{T}(r_h, u_h, v_h).
\end{equation}
Combining \eqref{eq:Snew}, \eqref{eq:S1}, and \eqref{eq:S2}, and using \eqref{4.17} to express boundary integrals in terms of averages and jumps, we complete the proof of Lemma \ref{lemma:nonlinear-coupling}.

\section{Implementation details}\label{appendix:implementation}
In this part, we describe the implementation of the conservative DG methods for the gKdV equation and the HS-KdV system. Our numerical experiments in Section \ref{sec:numerics} are performed using the IRK4 scheme \eqref{eq:IRK4} as the time-stepping method. For clarity of presentation, we describe the implementation of  our conservative DG methods using the Implicit Midpoint rule  \eqref{2.11}-\eqref{2.12}, while noting that the extension to IRK4 follows analogously with multiple stage variables.

\subsection{Implementation of the conservative DG for the gKdV equation}\label{sec:implementation_kdv}
When applying the Implicit Midpoint rule \eqref{2.11}--\eqref{2.12} to our DG formulation defined by \eqref{DEqs}, \eqref{NTU}, and \eqref{EH}, we obtain the fully discrete scheme, which can be written in matrix–vector form:
\begin{subequations}
	\begin{align}
		&M[q_h] + (D+A)[u_h] = 0, \label{2.15a}\\
		&M[p_h] + \epsilon (D+A)[q_h] - M[f(u_h)] = 0, \label{2.15b}\\
		&M[u_h] - \frac{\Delta t}{2}(D+A)[p_h]
		- \frac{\Delta t}{2}\tau_{pq} J[q_h]
		- \frac{\Delta t}{2}\tau_{pu} J[u_h]
		- M[\bar{u}_h]
		- \frac{\Delta t}{2} M[g] = 0, \label{2.15c}\\
		&\tau_{pu} \sum_{i=1}^{N} \jmp{u_h}^2(x_i)+\tau_{pq} \sum_{i=1}^{N} \jmp{u_h}\jmp{q_h}(x_i)-\sum_{i=1}^{N} \Big( \jmp{V(u_h)} - \{\Pi f(u_h)\}\jmp{u_h} \Big)(x_i)=0, \label{2.15d}\\
		& \tau_{pu}  \sum_{i=1}^N \jmp{p_h}\jmp{u_h}(x_i) +\tau_{pq} \sum_{i=1}^N \jmp{p_h}\jmp{q_h}(x_i)= 0.\label{2.15e}
	\end{align}
\end{subequations}
Here, the vectors \([u_h]\), \([q_h]\), and \([p_h]\) represent the degrees of freedom corresponding to the DG solutions \(u_h^{n+\frac{1}{2}}\), \(q_h^{n+\frac{1}{2}}\), and \(p_h^{n+\frac{1}{2}}\), respectively, and \([\bar{u}_h]\) denotes the known solution at the previous time level \(u_h^{n}\). The matrices \textbf{M}, \textbf{D}, \textbf{A}, and \textbf{J} denote the standard mass matrix, derivative operator, average flux, and  interface jump contributions, respectively.

Starting from the initial condition, \(u_h^0\) is known. Using \eqref{2.15a}, one computes \(q_h^0\), and with \(u_h^0\) and \(q_h^0\) available, \(p_h^0\) is obtained from \eqref{2.15b}. Once \(u_h^0\), \(q_h^0\), and \(p_h^0\) are determined, the penalization parameters \(\tau_{pq}^0\) and \(\tau_{pu}^0\) can be obtained by solving a $2\times 2$ linear system, \eqref{2.15d}–\eqref{2.15e}. At each time step, the method proceeds as follows:
\begin{itemize}
	\item a coupled nonlinear system is solved at the intermediate stage,
	\item the solution $u_h$ is updated,
	\item and remaining variables and penalty parameters are updated sequentially from linear solves.
\end{itemize}
This structure  reduces the number of nonlinear solves per time step compared to the previous framework.

The overall procedure is summarized in a flowchart in Fig. \ref{fig:1}. For IRK4, the only modification is that the midpoint stage is replaced by multiple intermediate stage solutions, after which the update and recovery steps proceed in the same way.
\begin{figure}[h]
	\centering
	\includegraphics[width=0.9\textwidth]{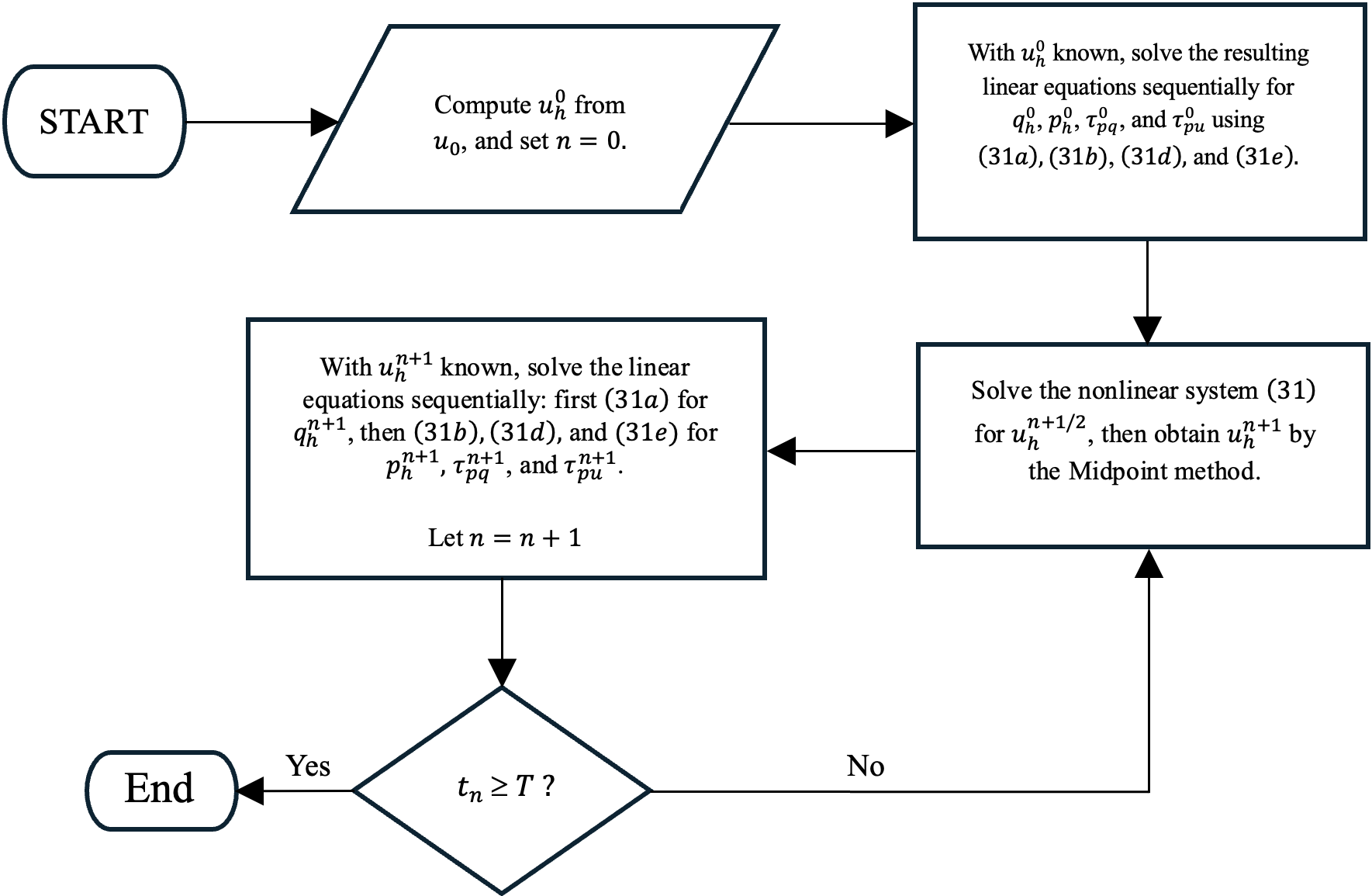}
	\caption{Flowchart illustrating the implementation of the conservative DG for the gKdV equation with the Implicit Midpoint rule.}
	\label{fig:1}
\end{figure}

\subsection{Implementation of conservative DG method for the HS-KdV system}
In this section, we again describe the implementation of the conservative DG method for the HS-KdV system using the Implicit Midpoint rule for ease of presentation. The fully discrete scheme can be written in the following matrix form:
\begin{subequations}\label{MV-CKdV}
	\begin{align}
		&  M[q_h] + (D+A)[u_h] = 0, \label{3.30a}\\
		\label{3.30b}
		& M[p_h] + (D+A)[q_h] - 3M[u_h^2] = 0, \\
		\label{3.30c}
		& M[u_h] + \frac{\Delta t}{2}(D+A)\left(a[p_h]+b[\Pi v_h^2]\right) - M[\bar{u}_h]
		+ a \frac{\Delta t}{2}\tau_{pu} J[u_h]
		+ a \frac{\Delta t}{2}\tau_{pv} J[v_h]=0,\\
		\label{3.30d}
		&  M[w_h] + (D+A)[v_h] = 0,\\
		\label{3.30e}
		&  M[r_h] + (D+A)[w_h] = 0,\\
		\label{3.30f}
		&  M[v_h] - \frac{\Delta t}{2}(D+A)[r_h] - M[\bar{v}_h] + \frac{\Delta t}{2}M[3 u_h w_h] = 0,\\
		\label{3.30g}  
		\begin{split}
			& \sum_{i=1}^N \Big[ 2a\tau_{pu}\jmp{u_h}^2 + 2a\tau_{pv}\jmp{u_h}\jmp{v_h} -2b\jmp{v_h^2}\{u_h\} + 2b\jmp{u_h}\{\Pi v_h^2\} -2b\jmp{u_h}\{v_h^2\} \\&\qquad -2a\jmp{V(u_h)}+2a\jmp{u_h}\{\Pi f(u_h)\} + 4b\jmp{v_h}\{\Pi (u_h v_h)\} \Big](x_i) = 0,
		\end{split}\\
		\label{3.30h}
		\begin{split}
			&   \sum_{i=1}^N \Big[\left(a\jmp{p_h} + b\jmp{\Pi v_h^2}\right)\left(a\tau_{pu}\jmp{u_h}+a\tau_{pv}\jmp{v_h} \right)+a\jmp{p_h}\left(\tau_{pu}\jmp{u_h}+\tau_{pv}\jmp{v_h}\right)\\
			&\qquad+2b\Theta(q_h,w_h,v_h)-2b\Theta(r_h,u_h,v_h)-2b\Theta(u_h,w_h,w_h)-\Theta(p_h,v_h,v_h)\Big] \\
            &+ 6b [\Pi u_h^2]^T M [\Pi (w_h v_h)] -6b [\Pi (u_h w_h)]^T M [\Pi (u_h v_h)]= 0.
		\end{split}
	\end{align}
\end{subequations}
where the function $\Theta$ is defined in \eqref{eq:theta}.

The solution procedure in \eqref{MV-CKdV} begins with the initial data. Since \(u_h^0\) and \(v_h^0\) are known, \(q_h^0\), \(p_h^0\), \(w_h^0\), and \(r_h^0\) are computed successively from \eqref{3.30a}, \eqref{3.30b}, \eqref{3.30d}, and \eqref{3.30e}. The two penalization parameters \(\tau_{pu}^0\) and \(\tau_{pv}^0\) are then determined from the $2\times 2$ linear system \eqref{3.30g}-\eqref{3.30h}. At the intermediate time level \(t_{n+\frac{1}{2}}\), the coupled equations \eqref{3.30a}-\eqref{3.30h} form a nonlinear system; the subsequent update at $t_{n+1}$ again reduces to a sequence of linear solves. 
Note that the number of equations in \eqref{3.30a}--\eqref{3.30f} is $6N(k+1)$, and  enforcing the additional constraints \eqref{3.30g} and \eqref{3.30h} for the conservation of energy and Hamiltonian only increases the size of the algebraic system by two.

\section{Acknowledgment}
The authors thank Rebecca Pereira for her contributions to the earlier developments of the conservation formulation.
This work was supported in part by the National Science Foundation under Grant DMS-2309670.



\end{document}